\documentclass{amsart}
\numberwithin{equation}{section}
\usepackage{amssymb}
\usepackage{amsmath}
\usepackage{amsthm}
\usepackage{amscd}
\usepackage{amsfonts}
\usepackage{ascmac}
\usepackage[usenames]{color}
\usepackage{graphics}
\usepackage{DotArrow}

\theoremstyle{plain}
\newtheorem{theorem}{Theorem}[section]
\newtheorem{proposition}[theorem]{Proposition}
\newtheorem{corollary}[theorem]{Corollary}
\newtheorem{lemma}[theorem]{Lemma}

\newtheorem{example}[theorem]{Example}
\newtheorem{definition}[theorem]{Definition}

\newtheorem{fact}[theorem]{Fact}
\newtheorem{remark}[theorem]{Remark}

\newcommand{\bfC}{{\mathbf C}}

\newcommand{\bfP}{{\mathbf P}}

\newcommand{\barj}{{\overline j}}

\newcommand{\bark}{{\overline k}}
\newcommand{\barl}{{\overline \ell}}

\newcommand{\barz}{{\overline z}}

\newcommand{\barw}{{\overline w}}

\newcommand{\barbet}{{\overline \beta}}

\newcommand{\barpsi}{{\overline \psi}}
\newcommand{\barvarphi}{{\overline{\varphi}}}
\newcommand{\barpartial}{{\overline \partial}}

\newcommand{\mapright}[1]{\smash{\mathop{   \hbox to 0.7cm{\rightarrowfill}}
  \limits^{#1}}}

\newcommand{\Ker}{{\rm Ker}}

\newcommand{\Ric}{\operatorname{Ric}}

\newcommand{\grad}{\mathrm{grad}}

\newcommand{\Fut}{\mathrm{Fut}}
\newcommand{\Vol}{\mathrm{Vol}}

\newcommand{\pa}{\partial}

\newcommand{\gm}{\gamma}
\newcommand{\ka}{K\"ahler }
\newcommand{\ke}{K\"ahler-Einstein }

\newcommand{\lb}{\left (}
\newcommand{\rb}{\right )}

\newcommand{\C}{{\mathbb C}}

\newcommand{\Aut}{\operatorname{Aut}}
\newcommand{\Isom}{\operatorname{Isom}}
\newcommand{\TildeT}{\widetilde{T}}

\title{On coupled K\"ahler-Einstein metrics and weighted solitons on Fano manifolds}
\author{Akito Futaki\\
\\
Celebrating Prof. Shing-Tung Yau's 75th birthday}
\address{Yau Mathematical Sciences Center, Tsinghua University, Haidian district, Beijing 100084, China}
\email{futaki@tsinghua.edu.cn}

\date{April 12, 2026}

\begin{document}

\begin{abstract}
We consider coupled K\"ahler-Einstein metrics and weighted solitons on Fano manifolds.
These are natural generalizations of K\"ahler-Einstein metrics.
As in the case of K\"ahler-Einstein metrics, the existence is known to be equivalent to algebraic
conditions which generalize the K-polystability.
In this survey, we outline recent developments for these two cases.
\end{abstract}

\maketitle


\section{Introduction}

Let $X$ be a compact K\"ahler manifold, that is,
for a Hermitian metric $g = (g_{i\barj})$ of $T^\prime X$, the K\"ahler form
$$ \omega_g = \sqrt{-1} \sum_{i,j =1}^n\ g_{i\barj}\ dz_i \wedge d\barz_j$$
satisfies the K\"ahler condition $d\omega_g = 0$.
The Ricci curvature of $g$ or $\omega$ is given by
\begin{equation*}
\mathrm{Ric}(g)_{i\barj} = - \frac{\partial^2}{\partial z_i\partial \barz_j} \log \det (g_{k\barl}).
\end{equation*}
The Ricci form of $g$ is given by
\begin{equation*}
\Ric(\omega) := \sqrt{-1} \sum_{i,j=1}^n\ \mathrm{Ric}(g)_{i\barj}\ dz_i \wedge d\barz_j 
= - \sqrt{-1} \partial \barpartial  \log \det (g_{k\barl}).
\end{equation*}
This is a real closed $(1,1)$-form. By the Chern-Weil theory we have $ c_1(X) = [\Ric(\omega)/2\pi]$. 
Calabi posed the problem to find a K\"ahler-Einstein metric 
$$ \mathrm{Ric}(g)_{i\barj} = \lambda\ g_{i\barj}$$
for $\lambda = -1, 0 \text{ or}\ 1$.
This amounts to solving the complex Monge-Amp\`ere equation with respect to the unknown function $\varphi$,
$$ \det (g_{i\barj} + \frac{\partial^2 \varphi}{\partial z^i\partial\overline{z^j}})/\det(g_{i\barj}) = e^{-\lambda \varphi + F}$$
for a given K\"ahler metric $g$ and a smooth function $F$.
By the theorem of Yau \cite{yau78}, if $c_1(X) < 0$ or $c_1(X) = 0$ then there exists a
K\"ahler-Einstein metric with $\lambda = -1$ or $\lambda = 0$. The case of $\lambda = 0$ gave a solution to
the Calabi conjecture.
The case of $c_1(X) < 0$ is also due to Aubin \cite{aubin76} independently.

For $\lambda = 1$, that is, when $c_1(X) > 0$ (i.e. $X$ a Fano manifold),
 Chen-Donaldson-Sun \cite{CDS3} and Tian \cite{Tian12} proved that the existence of K\"ahler-Einstein metrics
 on Fano manifolds is equivalent to a notion called K-polystability.  
 K-polystability uses Donaldson-Futaki invariant as the (infinite dimensional) GIT weight.
This assertion has been known as Yau-Tian-Donaldson conjecture.  

A $k$-tuple of K\"ahler metrics $\omega_1,\ \cdots, \omega_k$ on a compact K\"ahler manifold $X$
is called coupled K\"ahler-Einstein metrics if it satisfies
\begin{equation}\label{coupled KE}
\Ric(\omega_1) = \ \cdots = \Ric(\omega_k) = \lambda \sum_{i=1}^k \omega_i
\end{equation}
for $\lambda = -1,\ 0$ or $1$. The notion of such metrics were introduced  by Hultgren and Witt Nystr\"om \cite{HultgrenWittNystrom18}. 
If $\lambda = 0$ this is just a $k$-tuple of Ricci-flat metrics, and the existence for compact K\"ahler manifolds
follows from Yau's solution of the Calabi conjecture with $c_1(X) = 0$.

For $\lambda = -1$ or $\lambda = 1$ the existence problem is an extension for the problem
for negative or positive K\"ahler-Einstein metrics, and an obvious  necessary condition is $c_1(X) < 0$ or $c_1(X) > 0$. 
Hultgren and Witt Nystr\"om proved the existence of the solution of \eqref{coupled KE} for $\lambda = -1$ under the condition
$c_1(X) < 0$ extending Yau and Aubin, 
and many interesting results for $\lambda = 1$ under the condition $c_1(X) > 0$
including attempts to extend Chen-Donaldson-Sun and Tian. 

Further, Hultgren \cite{Hultgren17} proved the existence of coupled K\"ahler-Ricci solitons on toric Fano manifolds (Theorem \ref{toric existence} below), 
and showed
that there is a toric Fano manifold such that the classical Futaki invariant does not vanish but
for certain choice of decomposition of the anticanonical class the barycenter lies at the origin and thus, for such a choice, there
exists couple K\"ahler-Einstein metrics though there is no K\"ahler-Einstein metric, see Example \ref{example1} below.
Hence, coupled K\"ahler-Einstein metrics can be more flexible canonical K\"ahler metrics on Fano manifolds.

More recently, Fujita-Hashimoto \cite{fujitahashimoto24} gave an equivariantly uniform stability characterization of the existence for coupled K\"ahler-Einstein  metrics
with $\lambda = 1$, $c_1(X) > 0$. Naturally, the work of Fujita-Hashimoto extends earlier results for the uncoupled (i.e. $k=1$) K\"ahler-Einstein metrics
on Fano manifolds. Here is a brief summary of some of the well-known results in the uncoupled case.

In the case $k=1$ and $\lambda = 1$, i.e. in the case of positive K\"ahler-Einstein metrics, 
after the proof by Chen-Donaldson-Sun and Tian,  other proofs were also given by Datar-Sz\'ekelyhidi \cite{DatarSzeke16} and Chen-Sun-Wang \cite{CSW}.
While all of these proofs used the Gromov-Hausdorff convergence, a variational proof without using Gromov-Hausdorff convergence was further given by
Berman-Boucksom-Jonsson \cite{BBJ} under the condition of
uniform K-stability. In Berman-Boucksom-Jonsson's case, the automorphism group is discrete.

For non-discrete automorphism group case, 
Li \cite{LiChi22} showed the existence under the condition of $G$-uniform stability. 
The work of Liu-Xu-Zhuang \cite{LXZ22} shows that when $G$ contains the maximal
torus $G$-uniform stability is equivalent to K-polystability. 

The recently work of Fujita-Hashimoto extended the work of Li to the case of 
positive coupled K\"ahler-Einstein metrics. The first half of this survey will try to give an account of the paper of Fujita-Hashimoto and my joint works
\cite{FutakiZhang18}, \cite{FutakiZhang19} 
with Zhang. The main result of Fujita-Hashimoto which is an extension of Theorem 1.3 in Li \cite{LiChi22} is as follows.
\begin{theorem}[Fujita-Hashimoto \cite{fujitahashimoto24}]\label{fujita-hashimoto} Let $X$ be a Fano manifold, and $\mathrm{Aut}(X)$ be its
automorphism group which we assume is reductive. Let $\mathbb T$ be the maximal torus of $\mathrm{Aut}(X)$, and 
$\mathbb T_r \subset \mathbb T$ be the maximal compact torus.
Then the following are equivalent.
\begin{enumerate}
\item[(1)]\ \ $(X, \{L_i\}_{i=1}^k)$ is $\mathbb T$-reduced uniformly coupled Ding stable.
\item[(2)]\ \ $(X, \{L_i\}_{i=1}^k)$ admits a $\mathbb T_r$-invariant coupled K\"ahler-Einstein metrics.
\item[(3)]\ \ $(X, \{L_i\}_{i=1}^k)$ has vanishing coupled $\mathbb T$-Futaki character and $$\delta_{\mathbb T}^{\mathrm{red}}(X, \{L_i\}_{i=1}^k) > 1.$$
\end{enumerate}
\end{theorem}
Definitions of coupled $\mathbb T$-Futaki character  is given in \eqref{coupledFut}, and the definition of  $\delta_{\mathbb T}^{\mathrm{red}}(X, \{L_i\}_{i=1}^k)$ 
is given in \eqref{delta}.

Next, we turn to weighted solitons.
Let $X$ be a Fano manifold. 
As before we regard $2\pi c_1(X)$ as a K\"ahler class, and 
a K\"ahler form $\omega$ is expressed as
$ \omega = \sqrt{-1}\, g_{i\barj}\, dz^i \wedge dz^\barj$.
Let $\mathbb T_r$ be the real compact torus in the automorphism group $\Aut (X)$, and assume that $\omega$ is $\mathbb T_r$-invariant. 
Since $X$ is Fano and simply connected
the $\mathbb T_r$-action is Hamiltonian with respect to $\omega$. Since the $\mathbb T_r$-action naturally lifts to the anti-canonical line bundle
$K_X^{-1}$ we have a canonically normalized moment map $\mu_\omega : X \to \mathfrak t^\ast$
where $\mathfrak t$ is the Lie algebra of $\mathbb T_r$ and $\mathfrak t^\ast$ its dual space, c.f. Lemma 3.2 in \cite{futaki04}. 
Let
$\Delta := \mu_{\omega}(X)$ be the moment polytope. Then $\Delta$ is independent of $\omega \in 2\pi c_1(X)$.
Let $v$ be a positive smooth function on $\Delta$. Regarding $\mu$ as coordinates on $\Delta$ using the action angle
coordinates, we may sometimes write $v(\mu)$ instead of $v$. The pull-back $\mu^\ast_\omega v$ is a smooth function on $X$, and
for this we write $v(\mu_\omega)=\mu^\ast_\omega v = v \circ \mu_\omega$. 

We say that a K\"ahler metric $\omega$ in $2\pi c_1(X)$ a {\it weighted $v$-soliton} or simply {\it $v$-soliton} if
$$ \Ric(\omega) - \omega = \sqrt{-1} \partial\barpartial \log v(\mu_\omega).$$ 
We also call $\omega$ simply a {\it weighted soliton}
when it is $v$-soliton for some $v$, or when $v$ is obvious from the context.
Examples of weighted solitons are a K\"ahler-Ricci soliton when $v(\nu) = e^{\langle\nu,\xi\rangle}$ for some $\xi \in \mathfrak t$, 
a Mabuchi solitons when $v(\mu) = \langle\mu, \xi\rangle + a$ for some positive constant $a$, and a basic metric which 
which induces Calabi-Yau cone metrics outside the zero sections of the canonical line bundle and hence
a Sasaki-Einstein metrics on the $U(1)$-bundle of $K_X^{-1}$ when $v(\mu) = \lb \langle\mu,\xi\rangle + a\rb^{-m-2}$, see \cite{Inoue19}, \cite{Inoue22}, \cite{Lahdili18}, \cite{ACL21}, \cite{AJL21}, \cite{HanLi23}, 
\cite{LiChi21}.

The K-polystability characterization for K\"ahler-Einstein metrics is also extended for weighted solitons by
 Han-Li \cite{HanLi23}, Blum-Liu-Xu-Zhuang \cite{BLXZ} and Li \cite{LiChi21}. We apply it to extend a result of Cao-Sun-Zhang-Yau
\cite{CaoSunYauZhang2022} concerning deformations of K\"ahler-Einstein Fano manifolds to Fano manifolds with weighted solitons.

In this paper we consider the Kuranishi family 
$\varpi : \mathfrak X \to B$ of deformations of a Fano manifold $X$
which is a complex analytic family of Fano manifolds where $B$ is an open set in $\bfC^d$
containing the origin $0$ and we write $X_t:= \varpi^{-1}(t)$ and require $X_0 = X$, c.f. \cite{Kodaira86}, \cite{KodairaSpencer58}, \cite{Kuranishi64}, \cite{MorrowKodaira}, 
\cite{Sun12}, \cite{FSZ22}.
For a given K\"ahler form $\omega \in 2\pi c_1(X)$
let $f \in C^\infty (X)$ satisfy
\begin{equation*}\label{Kura1}
\Ric(\omega) - \omega = \sqrt{-1} \partial\barpartial f.
\end{equation*}
The Kuranishi family we consider in this paper is described by a family of vector valued $1$-forms 
parametrized by $t \in B$
$$\varphi(t) = \sum_{i=1}^k t^i\varphi_i + \sum_{|I|\ge2}t^I\varphi_I\ \in\ A^{0,1}(T^\prime X), \quad k = \dim H^1(X,\mathcal O(T^\prime X)),$$
such that 
\begin{equation}\label{Kura}
\begin{cases}
 \barpartial\varphi(t) = \frac12 [\varphi(t),\varphi(t)];\\
 \barpartial^\ast_f\, \varphi(t) = 0;\\
 \varphi_1, \cdots, \varphi_k\ \text{form a basis of the space of all}\\
 \qquad\qquad T^\prime X\text{-valued}\ \Delta_f\text{-harmonic}\ (0,1)\text{-forms}\\
\end{cases}
\end{equation}
where $\Delta_f = \barpartial_f^\ast\barpartial + \barpartial\,\barpartial_f^\ast$ is the weighted Hodge Laplacian with
$\barpartial_f^\ast$ the formal adjoint of $\barpartial$ with respect to the weighted $L^2$-inner product
$\int_X (\cdot,\cdot) e^f \omega^n$.
See \cite{FSZ22} for more detail about this Kuranishi family. We showed in \cite{FSZ22} that the K\"ahler form 
$\omega$ on $X_0 = X$ remains to be a K\"ahler form
on $X_t$. 
The second half of this paper is to survey the following result. We write $T$ instead of $\mathbb T_r$ to simplify the
notation.
\begin{theorem}[\cite{futaki24PAMQ}]\label{Main Thm} Suppose that $X_0 = X$ has a weighted $v$-soliton. 
Consider the Kuranishi family (\ref{Kura}) with $f=\log v(\mu_\omega)$. Then, shrinking $B$ if necessary,  the following statements are equivalent.
\begin{enumerate}
\item[(1)]  $X_t$ has a weighted $v$-soliton for all $t\in B$.
\item[(2)] $T$ is included in $\Aut(X_t)$, and for the centralizer $\Aut^T(X_t)$ of $T$ in $\Aut(X_t)$, $\dim \Aut^T(X_t) = \dim \Aut^T(X_0)$ for all 
$t \in B$.
\item[(3)] $T$ is included in $\Aut(X_t)$, and the identity component $\Aut_0^T(X_t)$ of $\Aut^T(X_t)$ is isomorphic to $\Aut_0^T(X_0)$ for all 
$t \in B$.
\end{enumerate}
\end{theorem}

\section{Coupled K\"ahler-Einstein metrics.}
The general idea of the variational approach is the following. 
On a finite dimensional variational problem, there is a unique critical point
if the functional is strictly convex and proper.

K\"ahler-Einstein problem is set to a variational problem for functionals
called Mabuchi functional or Ding functional on the space of K\"ahler metrics. 
To describe the properness, Aubin's $J$-functional (or Darvas $d_1$-distance) has been used to play the role of distance
on the space of K\"ahler metrics.
Uniform stability is a way to describe the properness by means of the infinitesimal behaviors at $\infty$ of
the functionals using
$$ \lim_{s \to \infty} \frac{F(s)}{s}$$
the slope or ``non-Archimedian functional'' $F^{\mathrm{NA}}$ of a functional $F$ along a ray $\varphi(s)$, $s \in [0,\infty)$.
The following test configurations play the algebraic geometric role of the rays.

\begin{definition}
For an ample line bundle $L$ over a projective variety  $X$, 
a test configuration of exponent $r$
is a normal polarized scheme $(\mathcal X, \mathcal L)$ with the following properties:
\begin{enumerate}
\item[(1)] There is a $\bfC^*$-action on ${\mathcal X}$ lifting to $\mathcal L$,
\item[(2)] There is a flat $\bfC^*$-equivariant morphism $\pi : {\mathcal X} \to \bfP^1$ for the standard $\bfC^*$-action on $\bfP^1$,
such that
over $\bfP^1-\{0\}$, $(\mathcal X, \mathcal L)$ is equivarianty isomorphic to 
$(X \times (\bfC^\ast \cup \{\infty\}), p_X^\ast L^r)$
with the trivial action on the first factor $X$.
\end{enumerate}
The central fiber $X_0 = \pi^{-1}(0)$ is possibly singular, and
$\bfC^*$-action induces a $\bfC^*$-action on the central fiber
$L_0 \to X_0 = \pi^{-1}(0)$. 

Moreover if there is a $\bfC^\ast$-action on the polarized manifold $(X,L)$
and $(\mathcal X, \mathcal L)|_{\mathbb P^1 - \infty} \cong (X,L) \times \mathbb C$ 
with the diagonal $\bfC^\ast$-action
then $(\mathcal X, \mathcal L)$ 
is called a product configuration. 
\end{definition}

In comparison with the variational approach, $t \in \mathbb C^\ast$ corresponds to
$$ s = -\log |t| $$
and thus, the metric behavior near the central fiber $X_0$ describes the behavior (i.e. slope) of the Mabuchi and Ding functionals
as $s \to \infty$. In fact, Phong-Sturm \cite{PH07} showed that a weak geodesic ray in the space of K\"ahler metrics in $c_1(L)$ is associated with any test configuration.


Take
ample $\mathbb Q$-divisors $L_1, \cdots , L_k$ with 
$$- K_X \sim_{\mathbb Q} L_1 + \cdots + L_k,$$
and K\"ahler metrics $\omega_i \in c_1(L_i)$. Here, for $\mathbb Q$-divisors $L$ and $L^\prime$,  $L \sim_\mathbb Q L^\prime$ 
means $mL$ and $mL^\prime$ are divisors for some $m \in \mathbb Z$ and $mL$ and $mL^\prime$ are linearly equivalent. 
According to Hultgren-Witt Nystr\"om, 
the Donaldson-Futaki  invariant $\mathrm{DF}(\mathcal X,(\mathcal L_i))$ for test configurations 
$\{(\mathcal X,\mathcal L_i)\}_{i=1}^k$
of $\{(X,L_i)\}_{i=1}^k$ is defined by

\begin{eqnarray*}
&&\mathrm{DF}(\mathcal X,(\mathcal L_i))   \\
&& =  \frac{1}{n+1}\sum_{i=1}^k \frac{ \mathcal L_i \cdots \mathcal L_i}{(\mathcal L_i|_1)^n} -
\frac1{(-K_{X})^n}{ (\sum\mathcal L_i +K_{\mathcal X/\bfC})( \sum\mathcal L_i) \cdots (\sum\mathcal L_i)}
\end{eqnarray*}
where $(\mathcal L_i|_1)^n$ and $(-K_{X})^n$ are the intersections on $X$, and
$\mathcal L_i \cdots \mathcal L_i$ and $ (\sum\mathcal L_i +K_{\mathcal X/\bfC})( \sum\mathcal L_i) \cdots (\sum\mathcal L_i)$ are the intersections on $\mathcal X$.
Note that Hultgren-Witt Nystr\"om assume that the total space $\mathcal X$ of the test configuration is the same for all $\mathcal L_i)$.

\begin{definition} [Case $c_1(X) > 0$,\ $k=1$, Tian, Donaldson]
$(X,L = -K_X)$ is K-polystable if $\mathrm{DF}(\mathcal X,\mathcal L))$ is 
non-negative for any test configurations, and equality holds only for product configurations.
\end{definition}
\begin{definition} [Case $c_1(X) > 0$,\ $k$ general, Hultgren-Witt Nystr\"om]
$(X,L_1 + \cdots + L_k = -K_X)$ is K-polystable if $\mathrm{DF}(\mathcal X,(\mathcal L_i))$ is 
non-negative for any test configurations, and equality holds only for product configurations.
\end{definition}

Hultgren-Witt Nystr\"om showed\\
--\ existence implies K-polystability, and \\
--\ uniqueness modulo connected group action (an extension of Bando-Mabuchi \cite{bandomabuchi87}, Berndtsson \cite{Bern15}).

Fujita-Hashimoto defined Ding invariant of $\{\mathcal X_i, \mathcal L_i\}$
using different choice of  $\mathcal X_i$'s as follows (extending the definition due to Berman \cite{Berman16}, Fujita \cite{fujita18}, \cite{fujita19} for $k=1$). 
\begin{definition}[Case $c_1(X) > 0$, k=1]\label{Ding1} Ding invariant $\mathrm{Ding}(\mathcal X, \mathcal L)$ of $(X, K_X^{-1})$ is defined by
$$ \mathrm{Ding}(\mathcal X, \mathcal L) = - \frac{\mathcal L^{n+1}}{(n+1)(-K_X)^n} - 1 
+ \mathrm{lct}(\mathcal X,\mathcal D_{\mathcal X,\mathcal L};\mathcal X_0) 
$$
where $\mathcal D_{\mathcal X,\mathcal L} \sim_\mathbb Q - \mathcal L - K_{\mathcal X/\mathbb P^1}$ supported on $\mathcal X_0$, and 
$$ \mathrm{lct}(\mathcal X,\mathcal D_{\mathcal X,\mathcal L};\mathcal X_0) = \sup \{t\ |\ (\mathcal X, \mathcal D_{\mathcal X,\mathcal L} + t \mathcal X_0) \text{\ log canonical}\}. $$
is the log-canonical threshold. 
Here log canonical means that $$\mathrm{coeff}_E (K_\mathcal Y - \mu^\ast (K_\mathcal X + \mathcal D_{\mathcal X,\mathcal L} + t \mathcal X_0) \ge -1$$
for  any prime divisor $E$ over $\mathcal X$ (i.e. $\mu : \mathcal Y \to \mathcal X$ is proper birational and $\mathcal Y$ regular and $E \subset \mathcal Y$).
\end{definition}
Ding invariant is the slope of the Ding functional, c.f. section 4. Berman \cite{Berman16} observed the appearance of the log-canonical threshold when he 
clarified the difference between the Donaldson-Futaki invariant and the
Ding invariant.

\begin{remark} With $k=1$ and general ample line bundle $L$, the Donaldson-Futaki invariant is used to test the existence of
cscK metrics in $c_1(L)$. On the other hand, with $k=1$ and $L = K_M^{-1}$ on Fano manifolds, the Ding invariant is used to test the
existence of K\"ahler-Einstein metrics in $c_1(M)$. The Donaldson-Futaki invariant is closely related with the slope of the Mabuchi functional.
See the equation (30) in Berman-Darvas-Lu \cite{BDL}.
Mabuchi functional and Ding functional are defined in section 4.
Donaldson-Futaki invariant and Ding invariant have the relation
$$
DF(\mathcal X, \mathcal L) \ge Ding(\mathcal X, \mathcal L),
$$
see Lemma 2.53 in Xu \cite{XuLN}, and 
coincide for test configurations with integral central fiber. 
For the product configurations they are equal to the
classical Futaki invariant \cite{futaki83.1} for the infinitesimal generator of the $\mathbb C^\ast$-action. See Remark \ref{Fut2}..
\end{remark}

\begin{definition} [k general, sum of test configurations by Fujita-Hashimoto] Let 
 \ $\mathbb T$ be the maximal torus in the reduced automorphism group. 
For any $1 \le i \le k$, take any ${\mathbb T}$-equivariant test
configuration $\pi : (\mathcal X_i,\mathcal L_i) \to \mathbb C$ of $(X, L_i$). 
Take a sufficiently divisible $r_0 \in r\mathbb Z_{>0}$. Set
$$\mathcal R^i := \bigoplus_{m \in r_0\mathbb Z_{\ge 0}} \mathcal R^i_m 
:= \bigoplus_{m \in r_0\mathbb Z_{\ge 0}} H^0(\mathcal X_i,m\mathcal L_i),
$$ 
$$
\mathcal R^i_m = \bigoplus_{\lambda \in \mathbb Z} t^{-\lambda}\mathcal F_i^{\lambda} R^i_m
$$
where 
$$\mathcal F_i^{\lambda} R^i_m := \mathcal F_{\mathcal X_i,\mathcal L_i}^{\lambda} H^0(X,mL) 
:= \mathrm{image} \left(H^0(\mathcal X_i, m\mathcal L_i)_\lambda \to H^0(X_1, mL_1)\right) $$
with $H^0(\mathcal X_i, m\mathcal L_i)_\lambda$ the $\lambda$-weight space of the $\mathbb C^\ast$-action.
Note by the 
Rees correspondence we have $(\mathcal X_i, \mathcal L_i) \cong (\mathrm{Proj}_{\mathbb C[t]}(\mathcal R^i), \mathcal O(1))$.
Let us take $(\mathbb C^\ast \times \mathbb T)$-equivariant common resolutions
$ \sigma_i : \mathcal Z \to \mathcal X_i$
with $\mathcal Z$ normal and $\sigma_i$ an isomorphism over $\mathbb C \backslash \{0\}$. 
We define
\begin{eqnarray*}
\mathcal R_m &:=&\mathrm{Image} (  \sigma_1^\ast \mathcal R^1_m \otimes_{{\mathbf C}[t]} \cdots
\otimes_{{\mathbf C}[t]}  \sigma_k^\ast \mathcal R^k_m \\
&\hookrightarrow & H^0(\mathcal Z, m\sigma_1^\ast \mathcal L_1) \otimes_{{\mathbf C}[t]} \cdots
\otimes_{{\mathbf C}[t]} H^0(\mathcal Z, m\sigma_k^\ast \mathcal L_k)\\
&\hookrightarrow& H^0(\mathcal Z, m(\sigma_1^\ast \mathcal L_1 + \cdots + \sigma_k^\ast \mathcal L_k)) )
\end{eqnarray*}
and
$$ \mathcal R :=\bigoplus_{m \in r_0\mathbb Z_{\ge 0}} \mathcal R_m 
 \subset \bigoplus_{m \in r_0\mathbb Z_{\ge 0}} H^0(\mathcal Z, m(\sigma_1^\ast \mathcal L_1 + \cdots + \sigma_k^\ast \mathcal L_k)). $$
 Then 
 $\mathcal R$ is a finitely generated $\mathbf C[t]$-algebra, and induces a $\mathbb T$-equivariant test configuration 
 $(\mathcal X,\mathcal L)$, called the sum configuration.
\end{definition}
\begin{definition}[Coupled Ding invariant by Fujita-Hashimoto]
Let $(\mathcal X,\mathcal L)$ be the sum configuration for $\{\mathcal X_i, \mathcal L_i\}$ as in the previous definition. Define the
coupled Ding invariant by
 $$\mathrm{Ding}(\{\mathcal X_i, \mathcal L_i\}) = \mathrm{Ding}(\mathcal X, \mathcal L) + \frac{\mathcal L^{n+1}}{(n+1)L^n} - \sum_{i=1}^k \frac{\mathcal L_i^{n+1}}{(n+1)L_i^n}.$$
\end{definition}

 \begin{definition}
 For $k=1$, replacing $\mathcal X$ by a birational model dominating both $\mathcal X$ and $X\times \mathbb P^1$, we define
 $$ {\mathbf J}(\mathcal X, \mathcal L) = \frac{\mathcal L\cdot \rho^\ast (L\times \bf P^1)^n }{L^n} - \sum \frac{\mathcal L^{n+1}}{(n+1)L^n}
$$
with $\rho : \mathcal X \to X \times \mathbb P^1$. 
${\mathbf J}(\mathcal X, \mathcal L))$ is the non-Archimedian $J$-functional. See section 4.
For the coupled case, define
$$ {\mathbf J}^\mathrm{cp}(\{\mathcal X_i, \mathcal L_i\}) := \sum_{i=1}^k {\mathbf J}(\mathcal X_i, \mathcal L_i),
$$
$$ {\mathbf J}^\mathrm{cp}_{\mathbb T}(\{\mathcal X_i, \mathcal L_i\})
  := \inf_{\xi \in N_{\mathbb Q}(\mathbb T)}\ \sum_{i=1}^k {\mathbf J}(\mathcal X_{i,\xi}, \mathcal L_{i,\xi})
$$
where $N_{\mathbb Q}(\mathbb T)$ is the rational points in the Lie algebra of $\mathbb T$ and $(\mathcal X_{i,\xi}, \mathcal L_{i,\xi})$ is the $\xi$-twist of $(\mathcal X_i, \mathcal L_i)$ obtained by twisting the $\mathbb C^\ast$-action
by the action induced by $\xi$.
\end{definition}

\begin{definition}[Fujita-Hashimoto]
We say $(X, \{L_i\}_{i=1}^k)$ is $\mathbb T$-equivariantly coupled Ding semistable if 
$$ \mathrm{Ding} \left(\{\mathcal X_i,\mathcal L_i\}_{i=1}^k   \right) \ge 0.$$
We say $(X, \{L_i\}_{i=1}^k)$ is $\mathbb T$-reduced uniformly coupled Ding stable if 
there exists a positive real number $\epsilon > 0$ such that 
$$ \mathrm{Ding} \left(\{\mathcal X_i,\mathcal L_i\}_{i=1}^k   \right) \ge \epsilon {\mathbf J}^{\mathrm cp}_\mathbb T (\{\mathcal X_i, \mathcal L_i\})).$$
\end{definition}
\begin{theorem}[= Theorem \ref{fujita-hashimoto}, Fujita-Hashimoto]
Let $\mathbb T_r \subset \mathbb T$ be the maximal compact torus of $\mathrm{Aut}(X)$. Then the following are equivalent.
\begin{enumerate}
\item[(1)] $(X, \{L_i\}_{i=1}^k)$ is $\mathbb T$-reduced uniformly coupled Ding stable.
\item[(2)] $(X, \{L_i\}_{i=1}^k)$ admits a $\mathbb T_r$-invariant coupled K\"ahler-Einstein metrics.
\item[(3)]$(X, \{L_i\}_{i=1}^k)$ has vanishing coupled $\mathbb T$-Futaki character 
and $$\delta_{\mathbb T}^{\mathrm{red}}(X, \{L_i\}_{i=1}^k) > 1.$$
(Definitions of coupled $\mathbb T$-Futaki character  is given in \eqref{coupledFut}, and the definition of  $\delta_{\mathbb T}^{\mathrm{red}}(X, \{L_i\}_{i=1}^k)$ 
is given in \eqref{delta}.)
\end{enumerate}
\end{theorem}

\section{Coupled $\mathbb T$-Futaki character.}

Now we suppose we have a decomposition $2\pi c_1(X) = \gm_1 + \cdots + \gm_k$
with $\omega_i \in \gm_i$. Since both $\Ric(\omega_i)$ and $\sum_{j=1}^k$ represent
$c_1(M)$ there exist real smooth functions $f_i$'s such that 
\begin{equation}\label{coupled}
\Ric(\omega_i) - i\partial\barpartial f_i =  \sum_{j=1}^k \omega_j, \quad i = 1,\ \cdots,\ k.
\end{equation}
Since $\partial\bar{\partial}\log e^{f_i}\omega_i^n$ are equal for all i, by adding suitable constants to $f_i$'s
we may normalize $f_i$ so that 
\begin{equation}\label{normalization}
e^{f_1}\omega_1^n= \cdots = e^{f_k}\omega_k^n := dV.
\end{equation}
Let $\xi$ be a holomorphic vector field on $X$. 
Since a Fano manifold is simply connected there exist complex-valued smooth functions defined up to constant
$u_i$ such that 
\begin{equation}\label{Hamiltonian}
i_\xi\omega_i=\bar\pa(\sqrt{-1}u_i).
\end{equation}
\begin{theorem}[\cite{FutakiZhang18}]\label{FZ18} With the choice of $\omega_i \in \gm_i$ with the normalization 
\eqref{normalization} of $f_i$, the Lie algebra $\mathfrak{h}(X)$ of all holomorphic vector fields is isomorphic to the set of all $k$-tuples of complex valued smooth functions $(u_1, \cdots u_k)$ satisfying 
$$\grad^\prime_i u_i=\grad^\prime_j u_j,\ \ i, j=1, 2, \dots, k,$$
\begin{equation}\label{Th3.3}
\Delta_{i} u_i+(\grad_{i} u_i)f_i = -\sum\limits_{j=1}^k u_j.
\end{equation}
for $i=1, 2, \dots, k$ where $\Delta_{i}=-\bar\partial^*_{i}\bar\partial$ with respect $\omega_i$.
\end{theorem}
\begin{corollary}[\cite{FutakiZhang18}, \cite{HultgrenWittNystrom18}]\label{FZ18-2}
If a Fano manifold $X$ admits coupled \ke metrics, then the Lie algebra $\mathfrak{h}(X)$ of holomorphic vector fields is reductive. 
\end{corollary}
\begin{theorem}[\cite{FutakiZhang18}]\label{FZ18-3}
$\mathfrak h(X)$ is isomorphic to\\
$\{u_1+ \cdots + u_k\ |\ \grad^\prime_i u_i=\grad^\prime_j u_j \in \mathfrak h(X),\ i, j=1, \dots, k, 
\int_X (u_1 + \cdots + u_k)\ dV = 0 \}$.
Further, 
if we replace $u_i$ by $u_i^{c_i}=u_i+c_i$ the equations \eqref{Th3.3} are satisfied for $u_i^{c_i}$ 
if and only if 
\begin{equation}\label{ambiguity} 
\sum\limits_{i=1}^k c_i=0.
\end{equation}
The normalization condition $\int_X (u_{1,V} + \cdots + u_{k,V})\ dV = 0$ is equivalent to
$$ \sum_{i=1}^k \mathcal P_i =\mathcal P_{-K_X}$$
where the left side hand is the Minkowski sum of the moment polytopes $\mathcal P_i$'s for $\gm_i$'s
and $\mathcal P_{-K_X}$ is the standard moment image of $-K_X$. 
\end{theorem}
Using these isomorphisms of $\mathfrak h(X)$
we define
\begin{align}
\Fut: \mathfrak{h}(X)&\to \C\nonumber\\
\xi&\mapsto \sum\limits_{i=1}^k\frac{\int_X u_i\ \omega_i^n}{\int_X \omega_i^n} \label{coupledFut}
\end{align}
where $\xi = \grad^\prime_i u_i$ for every $i =1, \dots, k$.
\begin{theorem}[\cite{FutakiZhang18}]\label{FZ18-4}
Let $X$ be a Fano manifold.
\begin{enumerate}
\item[(1)]$\Fut(\xi)$ is independent of the choice of \ka forms $\omega_i\in \gm_i$, $i=1, \dots, k$.
\item[(2)] If $X$ admits coupled K\"ahler-Einstein metrics then $\Fut= 0$.
\end{enumerate}
\end{theorem}

To show (2) in the above theorem, if we assume there exist coupled K\"ahler-Einstein metrics then we can take 
$f_1=\cdots=f_k=0$. Then by \eqref{normalization} and the normalization condition in Theorem \ref{FZ18-3} we see $\Fut(\xi) = 0$.

\begin{remark}\label{Fut2}
For $k=1$, the formula \eqref{coupledFut} coincides with the first line in page 55 of \cite{futaki88} up to sign.
In view of the theory K-stability, the opposite sign of \eqref{coupledFut} gives the right sign.
The Ding invariant in Definition \ref{Ding1} for the product configuration coincides with
the last term in (5.2.1) in \cite{futaki88}. Thus by the equations there, Ding invariant and 
the Donaldson-Futaki invariant for the product configurations coincide with the original Futaki
invariant.
Changing the sign to the standard convention, we see
using \eqref{Th3.3} 
\begin{equation}\label{Fut3}
\Fut(\grad u) = \int_X (\grad\, u)f\ \omega^n.
\end{equation}
On a compact K\"ahler manifold $(X,\omega)$, writing the scalar curvature by $S(\omega)$
and its average by
$\underline{S}=\frac{1}{V} \int_M S(\omega)\, \omega^m$, we may define $f$ by
$$ S(\omega) - \underline{S} = \Delta f$$
and extend $\Fut$ by
\begin{eqnarray}
\Fut(\grad\, u) &=& \int_X (\grad u)f \omega^n \nonumber\\
&=& - \int_X u(S(\omega) - \underline{S}) \omega^n. \label{Fut4}
\end{eqnarray}
This last expression is more useful for the study of cscK case, see section \ref{section functionals}.
\end{remark}

As a convenient tool to compute the invariant $\Fut$ we formulated in \cite{FutakiZhang19} a localization formula
extending \cite{futakimorita85}.
To formulate it, let $Z=\bigcup\limits_{\lambda\in\Lambda}Z_\lambda$ be zero sets of a holomorphic vector field $\xi$. 
Let $N_i(Z_\lambda)\cong (TX|_{Z_\lambda})/TZ_\lambda$ be the normal bundle of $Z_\lambda$ with respect to $\omega_i$. 
Then the Levi-Civita connection $\nabla^i$ of $\omega_i$ naturally induces an endomorphism $L^{N_i}(\xi)$ of $N_i(Z_\lambda)$ by 
\[L^{N_i}(\xi)(\eta)=(\nabla^i_\eta \xi)^{\perp}\in N_i(Z_\lambda), \quad \text{ for any } \eta\in N_i(Z_\lambda).\]
This definition of $L^{N_i}$ is independent of $\nabla^i$, and can be defined without using $\nabla^i$.
We also assume $Z$ is non-degenerate in the sense that $L^{N_i}$ is non-degenerate.
Let $K_i$ be the curvature of $N_i(Z_\lambda)$.
The localization formula of $\Fut(\xi)$ we obtain is the following.
\begin{theorem}\label{localization}
Let $X$ be a Fano manifold with $K_X^{-1}=L_1\otimes\dots\otimes L_k$. Let $\xi$ be a holomorphic vector field with non-degenerate zero sets $Z=\bigcup_{\lambda\in\Lambda}Z_\lambda$, then
\begin{eqnarray*}
&&(n+1)\Fut(\xi)=\\
&&\sum\limits_{i=1}^k\Bigg(\frac{\sum\limits_{\lambda\in \Lambda}\int_{Z_\lambda}\big(\big(E_i+c_1(L_i)\big)|_{Z_\lambda}\big)^{n+1}\big /\det\big((2\pi)^{-1}(L^{N_i}(X)+\sqrt{-1}K_i)\big)}{\sum\limits_{\lambda\in \Lambda}\int_{Z_\lambda}\big(\big(E_i+c_1(L_i)\big)|_{Z_\lambda}\big)^{n}\big /\det\big((2\pi)^{-1}(L^{N_i}(X)+\sqrt{-1}K_i)\big)}\Bigg).
\end{eqnarray*}
where $E_i\in \Gamma(End(L_i))$ is given by $E_i s = u_i s$ with $L^{N_i}$ and $K_i$ being as above.
\end{theorem}

\begin{corollary}
If $Z$ consists only of discrete points, then
\begin{eqnarray*}
\Fut(\xi)&=&\frac{1}{n+1}\sum\limits_{i=1}^k\Bigg(\frac{\sum\limits_{p\in Z}(u_i(p))^{n+1}/\det(\nabla X)(p)}{\sum\limits_{p\in Z}(u_i(p))^{n}/\det(\nabla X)(p)}\Bigg)\\
&=&\frac{1}{n+1}\Bigg(\sum\limits_{i=1}^k\frac{\sum\limits_{p\in Z}(u_i(p))^{n+1}}{\sum\limits_{p\in Z}(u_i(p))^{n}}\Bigg).
\end{eqnarray*}
\end{corollary}

\begin{example}\label{example1}
Consider the tautological line bundles $\mathcal{O}_{\mathbb{CP}^1}(-1)\to \mathbb{CP}^1$, 
 $\mathcal{O}_{\mathbb{CP}^2}(-1)\to \mathbb{CP}^2$, and also the bundle 
 $$E=\mathcal{O}_{\mathbb{CP}^1}(-1)\oplus\mathcal{O}_{\mathbb{CP}^2}(-1) \to  \mathbb{CP}^1\times\mathbb{CP}^2.$$
 Let $X$ be the total space of the projective line bundle $\mathbb{P}(E)$ over $\mathbb{CP}^1\times\mathbb{CP}^2$. \\
 (i)\ \ 
In \cite{futaki83.1} it was shown that 
$X$ is Fano, that $\operatorname{Aut}(X)$ is reductive, but that the Futaki invariant (\ref{Fut3}) or (\ref{Fut4}) as an obstruction to the existence of K\"ahler-Einstein metrics does
not vanish. Therefore $X$ does not admit a K\"ahler-Einstein metric.\\
(ii)\ \ 
Hultgren (\cite{Hultgren17}, Corollary 4) showed the vanishing of
the Futaki invariant (\ref{coupledFut}) as an obstruction
to the existence of coupled K\"ahler-Einstein metrics for certain choice of decomposition $K_X^{-1} = L_1\otimes L_2$.\\
(iii)\ \  On the other hand Hultgren (\cite{Hultgren17}, Theorem 2) showed that, on toric Fano manifolds, the existence of coupled K\"ahler-Einstein
metrics is equivalent to the vanishing of the Futaki invariant (\ref{coupledFut}) as an obstruction to the existence of coupled K\"ahler-Einstein metrics. \\
Since $X$ is toric, by (ii) and (iii) above, $X$ admits coupled K\"ahler-Einstein metrics although $X$ does not admit a K\"ahler-Einstein metric. This shows that coupled
K\"ahler-Einstein metrics may also play a good role of canonical K\"ahler metrics on Fano manifolds.

Below, we use our residue formula Theorem \ref{localization} to recover the result of Hultgren stated in (ii) above,
Let $D_4$ and $D_5$ be the divisors corresponding to the boundaries of the moment image of $\mathbb{CP}^1$.
We consider the following decomposition for $c\in(1/4, 3/4)$ which is ampleness condition
for the line bundles associated with $D(c)$ and $D(1-c)$ below.  Define
\begin{align*}
D(c)&=\frac{1}{2}K_M^{-1}+(c-\frac{1}{2})(D_4+D_5)\\
 D(1-c)&=\frac{1}{2}K_M^{-1}+(\frac{1}{2}-c)(D_4+D_5),
\end{align*}
then
\begin{equation}\label{decomp}
K_X^{-1}=D(c)+D(1-c).
\end{equation}
We take $\xi$ to be the holomorphic vector field along the fiber of $\mathbb{P}(E)$.
Using the localization formula,  we obtain
\begin{eqnarray*}
\Fut(\xi)&=&\frac{\Big[\frac{(u|_{Z_\infty}+c_1(D(c))|_{Z_\infty})^5}{u|_{Z_\infty}+c_1(\nu(Z_\infty))}+\frac{(u|_{Z_0}+c_1(D(c))|_{Z_0})^5}{u|_{Z_0}+c_1(\nu(Z_0))}\Big][\mathbb{CP}^1\times\mathbb{CP}^2]}{\Vol(D(c))}\\
&&+\frac{\Big[\frac{(u|_{Z_\infty}+c_1(D(1-c))|_{Z_\infty})^5}{u|_{Z_\infty}+c_1(\nu(Z_\infty))}+\frac{(u|_{Z_0}+c_1(D(1-c))|_{Z_0})^5}{u|_{Z_0}+c_1(\nu(Z_0))}\Big][\mathbb{CP}^1\times\mathbb{CP}^2]}{\Vol(D(1-c))}\\
&=&\frac{-15(112c^2-112c+23)}{(56c-3)(56c-53)},
\end{eqnarray*}
therefore, the invariant $\Fut$ vanishes when 
\[c=\frac{1}{2}\pm\frac{1}{4}\sqrt{\frac{5}{7}}.\]

\end{example}

So far we considered coupled K\"ahler-Einstein metrics, but
it is possible to consider coupled K\"ahler-Ricci solitons.
Suppose $c_1(M)>0$. For a decomposition of $ c_1(M)$
\begin{equation}\label{Sdecomp}
 c_1(M) = (\gm_1 + \cdots + \gm_k)/2\pi
\end{equation}
with K\"ahler classes $\gm_1,\ \cdots,\ \gm_k$, (where $\gm_j$ may or may not be the first Chern class of an ample $\mathbb Q$-line bundle,)
K\"ahler forms $\omega_i$ representing $\gm_i$ are called coupled K\"ahler-Ricci solitons if they satisfy
\begin{equation*}
\Ric(\omega_i) - i\partial\barpartial f_i =  \sum_{j=1}^k \omega_j, \quad i = 1,\ \cdots,\ k,
\end{equation*}
where $f_i$'s are Killing potential for $\omega_i$, i.e. $Jdf_i$ is a Killing vector field for $\omega_i$.
Note a Killing vector field on a compact K\"ahler manifold generates a biholomorphic automorphisms.
Hultgren (\cite{Hultgren17}) showed the existence of  
coupled K\"ahler-Ricci solitons on toric Fano manifolds. In fact, Hultgren's result (\cite{Hultgren17}, Theorem 2) stated in (iii) above
follows from this existence result of coupled K\"ahler-Ricci solitons. 
We can extend these results of Hultgren to Sasakian setting. 
Recall here that a compact Sasaki manifold $(S,g)$ of real dimension $2n+1$ is characterized by its 
Riemannian cone manifold $C(S)$ being a K\"ahler manifold of complex dimension $n+1$. 
A Sasaki manifold is said to be toric if the the K\"ahler cone $C(S)$ is toric, that is, if $C(S)$ admits an effective $(\bfC^\ast)^{n+1}$-action. 
We then obtain the following.
\begin{theorem}[\cite{FutakiZhang18}]\label{toric existence}
Let $S$ be a compact toric Sasaki manifold with positive basic first Chern class $c_1^B(S)$ with decomposition 
satisfying \eqref{Sdecomp} for transverse K\"ahler classes $\omega_i$.  Then $S$ admits coupled Sasaki-Einstein metrics for the decomposition \eqref{Sdecomp}
if and only if $\Fut$ identically vanishes.
\end{theorem}
\noindent
Theorem \ref{toric existence} follows from an existence result of toric coupled Sasaki-Ricci solitons.

\section{Ding functional and Ding invariant}\label{section functionals}

In this section we review on the Ding functional, the Ding invariant and the equivalence of (1) and (2) in Theorem \ref{fujita-hashimoto}.
We first review the cscK case and the Mabuchi functional. 

Let 
$\omega$ be a fixed K\"ahler form, and 
$$
\mathcal{H}_\omega=\left\{v \in C^{\infty}(M) \mid \omega_v:=\omega+i\partial \bar{\partial} v>0\right\}
$$
the space of K\"ahler forms cohomologous to $\omega$.
The
Monge-Ampere energy 
is defined by
$$
\operatorname{E}_\omega (u)=\operatorname{E}(u)=\frac{1}{(n+1) V} \sum_{j=0}^n \int_M u\, \omega_u^j \wedge \omega^{n-j} 
$$
where
$$
 u \in \mathcal{H}_\omega , \quad V=\int_M \omega^n.
$$
The pluripotential theory by Bedford-Taylor \cite{BedfordTaylor} states that
for non-smooth plurisubharmonic (PSH for short) functions $v$, which are subharmonic along any complex lines locally, the wedge product of currents
$i\partial\overline{\partial}v \wedge \cdots \wedge i\partial\overline{\partial}v$ makes sense. 
Similarly, we may consider $\omega$-plurisubhamonic functions $u$, i.e. when $\omega = i\partial\bar{\partial}\varphi$, $\varphi + u$ is subharmonic along any complex lines locally, and put
$$
\begin{aligned}
	& \mathcal E=\left\{u \in \operatorname{PSH}(X, \omega) \mid \int_M \omega_u^n=V\right\} \\
	& \mathcal E^{1}=\left\{u \in \mathcal E \mid \int_M |u|\, \omega_u^n\leq \infty\right\}.
\end{aligned}
$$
The spaces $\mathcal E$ and $\mathcal E^{1}$ were introduced by Guedj-Zeriahi \cite{GuedjZeriahi}. 
Boucksom-Eyssidieux-Guedj-Zeriahi \cite{BEGZ10} and Berman-Boucksom-Guedj-Zeriahi \cite{BBGZ13} extended
the functional $E$ to $\mathcal E^1$. Darvas \cite{Dar19} defined $d_1$-distance and proved
that the completion of $\mathcal H_{\omega}$ with respect to $d_1$ is $\mathcal E_1$ and that the functional $E$ is $d_1$-continuous.

We also have the following functionals. \\
(i)\ \ Aubin $I$-functional
	$$\begin{aligned}
		& I\left(u_0, u_1\right)=\frac{1}{V} \int_M\left(u_0-u_1\right)\left(\omega_{u_1}^n-\omega_{u_0}^n\right) \\
		& u_{0}, u_{1} \in \mathcal{H}_\omega\left(\text{or}\quad \mathcal E^1\right)
	\end{aligned}
	$$
(ii)\ \ Aubin $J$-functional $$
	J(u)=\frac{1}{V} \int_M u\, \omega^n -E(u)
	$$
(iii)\ \ $K$-energy (Mabuchi functional)
$$
	M(u)=\frac{1}{V} \int_0^1 d t \int_M \dot{v}_t\left(S\left(\omega_{v_t}\right)-\underline{S}\right) w_{v_t}^n
	$$
where $\underline{S}=\frac{1}{V} \int_M S(\omega)\, \omega^n$ is the average of the scalar curvature, and 
	$\left\{v_{t} \mid 0 \leq t \leq 1\right\}$ is a smooth path between $0$ and $u$
	such that $\omega_{v_{t}}>0$. This is inspired by the expression \eqref{Fut4}. 
	Mabuchi \cite{mabuchi87} showed that $M(u)$ is independent of the path $v_t$.
	
The following are also well known.\\
(iv)\ \ 
 $$
		 \frac{1}{n+1} I(u, 0) \leq J(u) \leq I(u, 0).
	$$
(v)\ \ A critical point of $M$ is a cscK metric.\\
(vi)\ \ If $\omega_{u_s}=\exp (sX)^* \omega$ for a holomorphic vector field $X$ then
		$$
		\left.\frac{d}{d s}\right|_{s=0} M\left(u_s\right)=\operatorname{Fut}(X).
		$$
Thus, $M$ is a non-linear version of the Futaki invariant. Note that (vi) implies that $M$ is $\operatorname{Aut}(X)$-invariant if and only if $\Fut$ vanishes.
Mabuchi energy has a decomposition, called Chen-Tian formula, 
	$$
		M(u)=H(u)+\underline{S} E(u)-n E_{\operatorname{Ric}(\omega)} (u), $$
where $H$, called the entropy, is defined by
$$	
		H(u)=V^{-1} \int_M \log \left(\omega_u^n / \omega^n\right) \omega_u^n,$$
and $E_{\operatorname{Ric}(\omega)}$ is defined by 
$$
		E_{\operatorname{Ric}(\omega)}u)=\frac{1}{n V} \sum_{j=1}^{n} \int_M u\, \operatorname{Ric}(\omega) \wedge \omega_u^{j-1} \wedge \omega^{n-j}.
		$$
See section 3 in \cite{Chen00} for its derivation. The slopes of these functionals are known to be expressed by intersection formulas. See \cite{Dyre18}
for their derivations.

Now we turn to Ding functional. This functional is useful to study K\"ahler-Einstein problem on Fano manifolds.
	Let $X$ be a Fano manifold and $\omega \in c_1()$ fixed. We define 
	$$
	D(u)=-\frac{1}{V} E(u)-\log \int_X e^{-u} d\mu'
	$$
	where $d\mu'$ is a probability measure on $X$, i.e. a Hermitian metric of $K_X^{-1}$ and $\int_X d\mu' = 1$,  such 
	$$ - i\partial\bar{\partial} \log d\mu' = 2\pi \omega.$$
Notice that $ d\mu'  = dV/\int_X dV$
where $dV$ is defined as in \eqref{normalization} with $k=1$. 

One can show that
 a critical point of $D$ is a K\"ahler-Einstein metric and that
	$$
	\left.\frac{d}{d t}\right|_{t=0} D\left(u_t\right)=\operatorname{Fut}(\xi) \qquad\text { if } \omega_{u_t}=\exp (t \xi)^* \omega.
	$$
This last equation shows that the Ding functional is invariant under the action of $\operatorname{Aut}(X)$ if and only if $\Fut$ vanishes.

Let us now define the coupled Ding functional $D^{cp}$ for a decomposition $K_X^{-1} = L_1 \otimes \cdots \otimes L_k$ of ample $K_X^{-1}$
into the tensor product of ample $\mathbb Q$-line bundles $L_i$. Let $dV$ be as in \eqref{normalization}, and put 
$$ d\mu'  = dV/\int_X dV.$$
Then \eqref{coupled} shows that the Ricci form of $d\mu'$ satisfies
$$\mathrm{Ric}(d\mu') = \omega_1 + \cdots + \omega^k.$$
Put
\begin{equation*}
\mathcal{H}_i=\left\{v_i \in C^{\infty}(M) \mid \omega_{v_i}:=\omega+i\partial \bar{\partial} v_i>0\right\}
\end{equation*}
and
\begin{equation*}
\mathcal{H}=\mathcal H_1 \times \cdots \times \mathcal H_k.
\end{equation*}
Let $\mathcal H_i^{\mathbb T_r}$ be the subset of $\mathcal H_i$ consisting of $\mathbb T_r$-invariant functions, and pu
\begin{equation*}
\mathcal{H}^{\mathbb T_r}=\mathcal H_1^{\mathbb T_r} \times \cdots \times \mathcal H_k^{\mathbb T_r}.
\end{equation*}
Define $D^{cp} : \mathcal H \to \mathbb R$ by
\begin{equation*}
D^{cp}(v_1, \cdots, v_k) = L^{cp}(v_1, \cdots, v_k) - \sum_{i=1}^k E_i(v_i)
\end{equation*}
where
\begin{equation*}
L^{cp}(v_1, \cdots, v_k) = - \log \int_X \exp\left(-\sum_{i=1}^k v_i\right)\ d\mu'
\end{equation*}
and
\begin{equation*}
\operatorname{E}_i(v_i)=\frac{1}{(n+1) \int_X c_1(L_i)^n} \sum_{j=0}^n \int_M v\, (\omega_i + i\partial\bar{\partial}v_i)^j \wedge \omega_i^{n-j} .
\end{equation*}
We say that the coupled Ding functional is $\mathbb T$-coercive if there exists $\epsilon > 0$ such that
\begin{equation*}
D^{cp}(v_1, \cdots, v_k) \ge \epsilon J_{cp, \mathbb T}(v_1, \cdots, v_k) - \frac1{\epsilon}
\end{equation*}
for all $(v_1, \cdots, v_k) \in \mathcal{H}^{\mathbb T_r}$ where 
\begin{equation}\label{J_cp}
J_{cp, \mathbb T}(v_1, \cdots, v) = \inf_{\sigma \in \mathbb T} \sum_{i=1}^k J_i(\sigma^\ast v_i)
\end{equation}
and
\begin{equation*}
J_i(v_i) = \frac1{\int_X c_1(L_i)^n} \int_X v_i\,\omega_i^n - E_i(v_i)
\end{equation*}
Note that the infimum in \eqref{J_cp} is attained for some $\sigma$ as shown in \cite{Hisamoto16} and \cite{LiChi22}, see also Lemma 3.5 in \cite{HanLi23}.

The proof of the equivalence of (1) and (2) in Theorem \ref{fujita-hashimoto} consists of the three steps below.

Step 1. (2) implies (1). That is, 
if $(X, \{L_i\})$ admits a $\mathbb T_r$-invariant coupled K\"ahler-Einstein metrics, then it is $\mathbb T$-reduced uniformly coupled Ding stable
(Theorem 8.5 in \cite{fujitahashimoto24}).

Step 2. $(X, \{L_i\})$ admits $\mathbb T_r$-invariant coupled K\"ahler-Einstein metrics if and only if $D^{cp}$ is $\mathbb T$-coercive
(Proposition 8.3 in \cite{fujitahashimoto24}). The proof relies on the properness criterion of Darvas-Rubinstein \cite{DR17}.

Step 3. If $(X, \{L_i\})$ is $\mathbb T$-reduced uniformly coupled Ding stable, then $D^{cp}$ is $\mathbb T$-coercive
(Theorem 8.10 in  \cite{fujitahashimoto24}).

That (1) implies (2) follows from Step 3 and Step 2.
\section{Valuative criterion and the delta invariant (stability threshold)}
In this section we define $\delta_{\mathbb T}^{\mathrm{red}}(X, \{L_i\}_{i=1}^k)$ in Theorem \ref{fujita-hashimoto}. We
begin with Fujita-Li's valuative criterion \cite{fujita19}, \cite{LiChi17}.

Let $E$ be a divisor over $X$, i.e. $\mu : Y \to X$ is proper birational and $Y$ regular and $E \subset Y$.
Define the log discrepancy by
$$ A_X(E) = 1 + \mathrm{coeff}_E (K_{Y/X}),$$
the ``expected vanishing order'' by
\begin{eqnarray*}
S_X(E) &=& \frac1{(-K_X)^n} \int_0^\infty \mathrm{vol}(-\mu^\ast K_X - tE)dt\\
(&=& \lim_{m\to\infty} \frac{\sum_{\lambda\in\mathbb Z}\lambda\dim\mathrm{gr}^\lambda_{F_v}H^0(X,-mK_X)}{m\dim H^0(X,-mK_X)}\quad )
\end{eqnarray*}
and the $\beta$-invariant by
$$ \beta_X(E) = A_X(E) - S(E).$$
By Li-Xu \cite{LX14}, to test K-(semi/poly)stability of a klt Fano variety $X$, it
suffices to consider special test configurations (i.e. the central fiber $\mathcal X_0$ is klt, normal, integral (meaning reduced and irreducible)). 
Any normal test configuration $(\mathcal X, \mathcal L)$ for a polarized pair $(X,L)$ induces a $\mathbb G_m$-equivariant
birational map $\mathcal X \dotarrow {}X \times \mathbb A^1$, and hence an isomorphism $K(\mathcal X) \cong K(X)(t)$.
If $v$ is a valuation of $K(X)(t)$, we denote by $r(v)$ its restriction to $K(X)$.
Each irreducible component $F \subset \mathcal X_0$ induces
a valuation $\mathrm{ord}_F$ of $K(\mathcal X) \cong K(X)(t)$ and hence $r(\mathrm{ord}_F)$ of $K(X)$.
It is known that $r(\mathrm{ord}_F)$ is a divisorial valuation, see Theorem 4.6 in \cite{BoucksomHisamotoJonsson}, or 
Proposition 2.2 in \cite{BlumLN}. It follows that,
for a special test configuration $(\mathcal X, \mathcal L)$ there is a divisor $E$ over $X$ such that
$$ r(\mathrm{ord}_{\mathcal X_0}) = c \mathrm{ord}_E$$
for some $c > 0$.
The following fact and its proof can be found in Lemma 4.19 in \cite{XuLN}, or Proposition 3.2 in \cite{BlumLN}.
\begin{fact}
Fact : If $(\mathcal X, -K_{\mathcal X/\mathbb A^1})$ is a test configuration with $\mathcal X_0$ integral, then
$$ DF(\mathcal X, -K_{\mathcal X/\mathbb A^1}) = c (A_X(E) - S_X(E))$$
for some $c > 0$.
\end{fact}
Motivated by this it was shown
\begin{theorem}\label{VC1}
With the notations being as above, 
\begin{enumerate}
\item[(1)] Fujita \cite{fujita19} - Li \cite{LiChi17}  : $X$ is K-semistable if and only if $A_X(E) - S_X(E) \ge 0$ for all divisors $E$ over $X$.
\item[(2)] Blum-Xu (\cite{BX19}) : $X$ is K-stable if and only if  $A_X(E) - S_X(E) > 0$ for all divisors $E$ over $X$.
\end{enumerate}
\end{theorem}

On the other hand, Fujita-Odaka \cite{fujitaodaka18} defined the delta invariant, or stability threshold as follows.
A $\mathbb Q$-divisor $0 \le D \sim_{\mathbb Q} -K_X$ is called an $m$-basis type if there exists a
basis $s_1, \cdots, s_{N_m}$ 
of $H^0(X,\mathcal O_X(-mK_X))$ such that
$$ D = \frac1{mN_m}(\{s_1 = 0\} + \cdots + \{s_{N_m} = 0\}).$$

\begin{definition}[Fujita-Odaka \cite{fujitaodaka18}]First define 
$$ \delta_m(X) := \mathrm{inf}\{\ \mathrm{lct}(X,D)\ |\ 0 \le D \sim_\mathbb Q -K_X \ \text{is}\ m\mathrm{-basis\ type}\}.$$
Then define $\delta$-invariant or stability threshold by
$$ \delta(X) = \lim\sup_{m\to\infty} \delta_m(X).  $$
\end{definition}
In fact, Blum-Jonsson \cite{BJ20} proved 
$$ \delta(X) = \inf_E \frac{A_X(E)}{S(E)}$$
where 
$E$ runs over prime divisors over $X$. One may compare this with Theorem \ref{VC1}.
\begin{theorem}[\cite{fujitaodaka18}, \cite{BJ20}, \cite{LiChi17}, \cite{LXZ22}]\label{VC2}
$X$ is K-semistable if and only if $\delta(X) \ge 1$, and 
$X$ is K-stable if and only if $\delta(X) > 1$.
\end{theorem}

As an extension to the coupled case, Fujita-Hashimoto \cite{fujitahashimoto24} defined the reduced delta invariant as follows.

\begin{equation}\label{delta}
 \delta^\mathrm{red}_{\mathbb T} (X;\{L_i\}) := \inf_{v \in \mathrm{Val}_X^{\ast,\mathbb T}}\ \sup_{\xi \in N_{\mathbb R}(\mathbb T)}\ 
\frac{A_X(v_\xi)}{\sum_{i=1}^k S_{L_i}(v_\xi)}, 
\end{equation}

$$ \mathrm{Val}^{\ast,\mathbb T}_X :=  \mathrm{Val}^{<\infty,\mathbb T}_X \backslash \{\mathrm{wt}_\xi\ |\ \xi \in  N_{\mathbb R}(\mathbb T)\}$$
where $\mathrm{Val}^{<\infty,\mathbb T}_X$ denotes the set of $\mathbb T$-invariant finite divisorial valuations of $X$,
$N_{\mathbb R}(\mathbb T)$ the Lie algebra of $\mathbb T$, and $v_\xi$ the $\xi$-twist of $v$, see (51) - (53) in \cite{LiChi22}.

Fujita-Hashimoto's proof of (1) implies (2) in Theorem \ref{fujita-hashimoto} relies on a work of Darvas-Rubinstein \cite{DR17} together with some analysis in \cite{HultgrenWittNystrom18} and \cite{LiChi22}.

As a related result, Kewei Zhang \cite{Kewei23}, \cite{Kewei24} introduced ``analytic delta invariant'' $\delta^{an}$ related to the Moser-Trudinger inequality,
showed $\delta^{an} = \delta$, and gave a new proof of Yau-Tian-Donaldson conjecture.


\section{Kodaira-Spencer deformation theory of compact complex manifolds}
Let us return to the Kuranishi family (\ref{Kura}) in section 1. Since $X$ is Fano we have
$H^2(X, \mathcal O (T'X)) = 0$
and then there exists $\varpi : \mathcal X \to B \subset \bfC^n \cong H^1(X, \mathcal O (T'X))$ such that,
writing $X_t:=\varpi^{-1}(t)$, we have
$$ \left\{ \left.\frac{\partial X_t}{\partial t}\right|_{t=0}\right\} \cong H^1(X, \mathcal O (T'X)) $$
where $\frac{\partial X_t}{\partial t}$ is the infinitesimal deformation, or the Cech coholomogy
class of the derivatives of the coordinate changes.
By the rigidity theorem of Kodaira-Spencer, 
small deformations of a compact K\"ahler manifold are K\"ahler.
There were studies on the K\"ahler geometry of deformations of K\"ahler manifolds by 
Sun, Yau, Zhang \cite{Sun12}, \cite{SunYau11}, \cite{ZhangYY2014}, and 
more recently by Cao-Sun-Yau-Zhang \cite{CaoSunYauZhang2022} which treats
the case when $X_0 = X$ is a K\"ahler-Einstein manifold.
Part of the work  \cite{CaoSunYauZhang2022} was extended by in our work \cite{FSZ22}
which treats the case when $X_0 = X$ is a general Fano manifold.
Further in \cite{futaki24PAMQ} the result for the case when $X_0 = X$ is a K\"ahler-Einstein manifold.
was extended to the case when $X=X_0$ is a Fano manifold 
with a weighted soliton (generalized
K\"ahler-Ricci soliton). The second half of this survey is about the results in \cite{FSZ22} and \cite{futaki24PAMQ}.

We start with preparatory results on eigenvalues of Hodge Laplacian in \cite{FSZ22}.
Let us recall the Kodaira vanishing.
Let $(X,g)$ be a compact K\"ahler manifold, 
$(L,h)$ be an Hermitian line bundle over $X$.
For the $\barpartial$-Laplacian $\Delta_{\barpartial}^L = \barpartial^\ast_L\barpartial + \barpartial\barpartial^\ast_L$ 
acting on an $L$-valued $(0,q)$-form $\eta$ we have
\begin{eqnarray*}
(\Delta_\barpartial^L\,\eta)_{\barj_1\cdots\barj_q} &=& - g^{i\barj}\nabla_i^L\nabla_\barj\eta_{\barj_1\cdots\barj_q} \\
&&\  -\ \sum_{\beta=1}^q (-1)^{\beta} g^{i\barj}[\nabla_i^L,\nabla_{\barj_\beta}]\eta_{\barj\barj_1\cdots\hat{\barj}_\beta \cdots \barj_q}. \nonumber
\end{eqnarray*}
Using the Ricci identity we obtain the following {\it Bochner-Kodaira formula}
\begin{eqnarray*}
(\Delta_\barpartial^L\,\eta)_{I\barj_1\cdots\barj_q} &=& - g^{i\barj}\nabla_i^L\nabla_\barj\eta_{\barj_1\cdots\barj_q}\\
&&\  + \sum_{\beta=1}^q g^{i\barj}(R_{i\barj_\beta} + \psi_{i\barj_\beta})\eta_{\barj_1 \cdots \barj_{\beta-1}\barj\barj_{\beta+1}\cdots \barj_q}.
\end{eqnarray*}
Hence, if $-K_X + L$ is positive then
$$H^q(X, \mathcal O(L)) = 0 \quad\text{for}\quad q > 0. $$
This is the proof of Kodaira vanishing.

Let $X$ be a Fano manifold of dimension $m$,
i.e. $2\pi c_1(X)$ is represented by a K\"ahler form.
Let $\omega$ be a K\"ahler form in $2\pi c_1(X)$, and 
$$ \Ric - \omega = \sqrt{-1}\partial\barpartial f ; $$
$f$ is called the Ricci potential.
Let $L=\mathcal O$ be the trivial line bundle with the Hermitian metric $e^f$. 
We write $\Delta_f := \Delta_\barpartial^L$ for our choice of the Hermitian metric $e^f$ on $L$. 
This is the same as considering the weighted volume form $e^f \omega^m$ for $(0,q)$-forms, 
and considering the weighted Hodge Laplacian $\Delta_f = \barpartial_f^\ast\barpartial + \barpartial\,\barpartial_f^\ast$ 
acting on differential forms of type $(0,q)$ where 
$\barpartial_f^\ast$ is the formal adjoint of $\barpartial$ with respect to 
the weighted $L^2$-inner product $\int_X (\cdot,\cdot) e^f \omega^m$.

For $\eta \in A^{0,q}(L) = A^{0,q}(X)$, and then, by the Bochner-Kodaira formula reads
\begin{eqnarray*}
(\Delta_f\,\eta)_{\barj_1\cdots\barj_q} &=& - g^{i\barj}\nabla_{i,f}\nabla_\barj\eta_{\barj_1\cdots\barj_q}\\
&&\  + \sum_{\beta=1}^q g^{i\barj}(R_{i\barj_\beta} - f_{i\barj_\beta})\eta_{\barj_1 \cdots \barj_{\beta-1}\barj\barj_{\beta+1}\cdots \barj_q} \nonumber\\
&=& - g^{i\barj}\nabla_{i,f}\nabla_\barj\eta_{\barj_1\cdots\barj_q} + q\ \eta_{\barj_1 \cdots \barj_q}
\end{eqnarray*}
where 
$$
\nabla_{i,f} = \nabla_i + f_i.
$$
Hence if $\Delta_f \eta = \lambda \eta$ then
$$ \lambda (\eta, \eta)_f = (\Delta_f \eta, \eta)_f = (\nabla^{\prime\prime}\eta, \nabla^{\prime\prime}\eta)_f + q(\eta,\eta)_f$$
and
$$ \lambda \ge q.$$
Further if
$\lambda = q$ then $\nabla^{\prime\prime}\eta = 0$. 
Since $\barpartial \eta$ is the skew-symmetrization of $\nabla^{\prime\prime}\eta$
it follows that $\barpartial \eta = 0$. 
Moreover, since $H^{0,q}_\barpartial(X) = 0$ for $q \ge 1$ on the Fano manifold $X$, $\eta$
is exact. 
Thus we have proved the following.
\begin{theorem}[\cite{FSZ22}]\label{eigen}Let $X$ be a Fano manifold and $\Delta_f$ be the weighted Hodge Laplacian as above. 
\begin{enumerate}
\item[(1)] If $\Delta_f \eta = \lambda \eta$ and $\eta \ne 0$ for a $(0,q)$-form $\eta$ then $\lambda \ge q$.
\item[(2)] If, in (1), $\lambda = q$ and $\eta \ne 0$ then $\nabla^{\prime\prime}\eta = 0$. 
In particular $\eta$ is closed, and for $q\ge 1$, $\eta$ is exact, 
and indeed, it is expressed as $\eta = \frac1{q}\barpartial(\barpartial_f^\ast\eta)$.
\end{enumerate}
\end{theorem}

Now, suppose $\Delta_f \eta = \lambda \eta$ and $\barpartial \eta \ne 0$. 
Since $\barpartial\Delta_f =\Delta_f\barpartial$ 
we have $\Delta_f \barpartial \eta = \lambda \barpartial \eta$. 
Apply (1) of Theorem \ref{eigen} to $\barpartial\eta$ which is non-zero $(q+1)$-form.
Then $\lambda \ge q+1$.
If $\lambda = q+1$ then by (2) of Theorem A we have $\nabla^{\prime\prime}\barpartial \eta =0$. 
This implies that $(\barpartial \eta)^\sharp$ is holomorphic. 
Thus we have proved the following.
\begin{theorem}[\cite{FSZ22}]\label{eigen2} Let $X$ be a Fano manifold and $\Delta_f$ be the weighted Hodge Laplacian as above. 
\begin{enumerate}
\item[(1)] If $\Delta_f \eta = \lambda \eta$ for a $(0,q)$-form $\eta$ and $\barpartial \eta \ne 0$ then $\lambda \ge q+1$.
\item[(2)] If, in (1), $\lambda = q+1$ then 
\begin{equation}
(\barpartial \eta)^\sharp := g^{i\barj}g^{i_1 \barj_1} \cdots g^{i_q\barj_q}\, \nabla_\barj\, \eta_{\barj_1 \cdots \barj_q}\, \frac{\partial}{\partial z^i} \wedge 
\frac{\partial}{\partial z^{i_1}} \wedge \cdots \wedge  \frac{\partial}{\partial z^{i_q}}
\end{equation}
is a holomorphic section of $\wedge^{q+1}T^\prime X$.
\end{enumerate}
\end{theorem}
The following is the special case of Theorem \ref{eigen2} for $q=0$ (and also of Theorem \ref{FZ18} for $k=1$). 
\begin{corollary}[\cite{futaki88}] Let $X$ be a Fano manifold, $\omega \in c_1(X)$ and $f$ the Ricci potential.
\begin{enumerate}
\item[(1)] If $\Delta_f u = \lambda u$ for a non-constant complex-valued smooth function $u$ then $\lambda \ge 1$.
\item[(2)] If, in (1), $\lambda = 1$ then $(\barpartial u)^\sharp$ is a holomorphic vector field.
\end{enumerate}
\end{corollary}

Now recall the deformation theory of Kodaira-Spencer.
Let $\varpi : \mathcal X \to B$ be a deformation of $X_0 := \varpi^{-1}(0) = X$.
Considering at $t=0$, for $t\in B$ small, 
$T^{\prime\ast}X_t$ and $T^{\prime\prime\ast}X_t$ are close to $T^{\prime\ast}X_0$ and $T^{\prime\prime\ast}X_0$. 
In particular, $T^{\prime\ast}X_t$ is expressed as a graph over $T^{\prime^\ast}X_0$,
and for all  $\alpha_t \in T^{\prime\ast}X_t$
$$\alpha'_t = \alpha'_0 + \varphi(t)(\alpha'_0)$$
where $\alpha'_0 \in T^{\prime\ast}X_0$ and $\varphi(t)(\alpha'_0) \in T^{\prime\prime\ast}X_0$.
Thus 
\begin{eqnarray*}
{}&&\varphi(t) \in \operatorname{Hom}(T^{\prime\ast}X_0,T^{\prime\prime\ast}X_0) = T^{\prime}X_0 \otimes T^{\prime\prime\ast}X_0, \nonumber\\
&& \quad \varphi(t) = \varphi^i{}_\barj (t) \frac{\partial}{\partial z^i} \otimes dz^\barj. 
\end{eqnarray*}
Here $z^i = z_0^i$ are local holomorphic coordinates of $X_0$.
Thus $T^{\prime\ast}X_t$ is spanned by
\begin{eqnarray*}
dz^i + \varphi^i{}_\barj(t) dz^\barj 
\end{eqnarray*}
or equivalently, $T^{\prime\prime}X_t$ is spanned by 
\begin{eqnarray*}
\frac{\partial}{\partial z^\barj} - \varphi^i{}_\barj(t) \frac{\partial}{\partial z^i} 
\end{eqnarray*}
Then 
$$
\begin{cases}
\barpartial\varphi(t) - \frac12 [\varphi(t),\varphi(t)] = 0;\\
\varphi(0) = 0;\\
 - \frac{\partial \varphi(t)}{\partial t}|_{t=0} =: \eta.\\
\end{cases}
$$
where $\barpartial \eta = 0$ and $[\eta] \in H^{0,1}_\barpartial(T'X)\cong H^1(X,\mathcal O(T'X))$.

For a Fano manifold $X$, we 
choose a K\"ahler form $\omega$ in $2\pi c_1(X)$ with the Ricci potential $f$ as before, and 
we consider the Kuranishi family described by
$$ \varphi(t) = \sum_{|I|=1} t^I\varphi_I + \sum_{|I|\geq 2} t^I\varphi_I$$
$$
\begin{cases}
\barpartial\varphi(t) = \frac12 [\varphi(t),\varphi(t)];\\
\barpartial_f^\ast \varphi(t) = 0;\\
\text{For}\  |I| = 1, \varphi_I\ \text{is}\ \Delta_f\text{-harmonic, so}\ \eta = \sum_{|I|=1} t^I\varphi_I\\
\end{cases}
$$
\begin{theorem}[\cite{FSZ22}]\label{deform1}
Let $X$ be a Fano manifold, $\omega$ a K\"ahler form in $2\pi c_1(X)$, 
and $\{X_t\}$ be the Kuranishi family of the deformation of complex structures described as above.
Then $\omega$ is a K\"ahler form on $X_t$ for any $t$.
\end{theorem}

\begin{proof}[Outline of the proof.] This extends a result of Sun \cite{Sun12} when $X_0$ is a K\"ahler-Einstein manifold
of positive Ricci curvature, and the proof follows similar arguments.
Since $d\omega = 0$ it is sufficient to show $\omega$ is $J_t$ invariant.
But 
$\omega$ is $J_t$ invariant if and only if
$\varphi_{\bark\barj}:=g_{i\bark}\varphi^i{}_\barj$ is symmetric in $j$ and $k$
since $T^{\prime\prime}X_t$ is spanned by 
$\frac{\partial}{\partial z^\barj} - \varphi^i{}_\barj(t) \frac{\partial}{\partial z^i}$).

Consider 
\begin{eqnarray*}
\varphi\lrcorner\,\omega &:=& \varphi^i{}_\bark dz^\bark \wedge \sqrt{-1} g_{i\barj} dz^\barj + \varphi^i{}_\barj dz^\barj \wedge \sqrt{-1} g_{i\bark} dz^\bark \\
&=&  - \sqrt{-1}(\varphi_{\barj\bark} - \varphi_{\bark\barj}) dz^\barj \wedge dz^\bark \nonumber\\
&=&- \sqrt{-1}\ \psi_{\barj\,\bark}\, dz^\barj\wedge dz^\bark, \nonumber
\end{eqnarray*}
where we have put 
\begin{equation*}
\psi_{\barj\,\bark} = \varphi_{\barj\,\bark} - \varphi_{\bark\,\barj}
\end{equation*}
which is the skew-symmetrization of $\varphi_{\barj\,\bark}$.
For the Kuranishi family above, we can show
$$ \Delta_f (\varphi\lrcorner\,\omega) =  \frac12 \varphi\lrcorner\,\omega.$$ 
Combining this with Theorem \ref{eigen}, we obtain $\varphi\lrcorner\,\omega =  0$. 
This implies $\varphi_{\bark\barj}$ is symmetric in $k$ and $j$, and hence $\omega$ is $J_t$-invariant.
This completes the proof of Theorem \ref{deform1}.
\end{proof}

\begin{theorem}[\cite{FSZ22}]\label{deform2} In Theorem \ref{deform1}, the Ricci potential of $(X_t,\omega)$ is given by
$$f + \log\det(1 - \varphi(t)\barvarphi(t))$$ 
up to an additive constant, more precisely,
$$
\Ric(X_t,\omega) = \omega + \sqrt{-1}\partial_t\barpartial_t (f + \log\det(1-\varphi(t)\barvarphi(t)))
$$
where $\Ric(X_t,\omega)$ denotes the Ricci form with respect to the complex structure $J_t$ on $X_t$.
\end{theorem}
\begin{proof}[Method of the proof.] This was essentially done in Zhang's thesis \cite{ZhangYY2014}. 
$T^{\prime\ast}X_t$ and $T^{\prime\prime}X_t$ are respectively spanned by
$$e^i := 
dz^i + \varphi^i{}_\barj(t) dz^\barj 
\quad \text{and} \quad 
T_\barj := \frac{\partial}{\partial z^\barj} - \varphi^i{}_\barj(t) \frac{\partial}{\partial z^i},\quad  
 i,\, j = 1, \cdots, m,$$ 
where $z^1, \cdots, z^m$ are local holomorphic coordinates for $X_0$.
Let $w^1, \cdots, w^m$ be local holomorphic coordinates for $X_t$ on the same coordinate neighborhood as $z^i$'s.
These two coordinates are related by
\begin{equation*}
\frac{\partial w^\beta}{\partial z^\barj} = \varphi^i{}_\barj\,\frac{\partial w^\beta}{\partial z^i} \quad \text{and} \quad
\frac{\partial z^i}{\partial w^\barbet} = - \varphi^i{}_\barj\,\frac{\partial z^\barj}{\partial w^\barbet}
\end{equation*}
from which we further obtain
\begin{equation*}
dw^\alpha = \frac{\partial w^\alpha}{\partial z^i}\, e^i
 \quad \text{and} \quad
\frac{\partial}{\partial w^\beta} = \frac{\partial z^j}{\partial w^\beta}\,T_j.
\end{equation*}
By computing
$\frac{\partial w^\alpha}{\partial w^\beta} = \delta^\alpha{}_\beta$ in terms of $z$ coordinates, we obtain
\begin{equation*}
\frac{\partial w^\alpha}{\partial z^i}\, (\delta^i{}_j - (\varphi\barvarphi)^i{}_j)\, \frac{\partial z^j}{\partial w^\beta} = \delta^\alpha{}_\beta.
\end{equation*}
We put 
$$ A = \left(a^\alpha{}_i\right), \quad a^\alpha{}_i = \frac{\partial w^\alpha}{\partial z^i}, \quad A^{-1} = \left(b^i{}_\alpha\right) $$
so that $b^i{}_\alpha a^\alpha{}_j = \delta^i{}_j$. Then for sufficiently small $t$
\begin{equation*}
A\,(I - \varphi\barvarphi)\,\frac{\partial z}{\partial w} = I, \quad \frac{\partial z^i}{\partial w^\alpha} = ((I - \varphi\barvarphi)^{-1})^i{}_j\,b^j{}_\alpha.
\end{equation*}
We also have
\begin{equation*}
\frac{\partial z^i}{\partial w^\barbet} = - \varphi^i{}_\barj\, \overline{((I - \varphi\barvarphi)^{-1})^j{}_\ell}\, \overline{b^\ell{}_\beta}.
\end{equation*}
Expressing $\frac{\partial}{\partial w^\alpha}$ in terms of $z$ coordinates we also obtain
\begin{equation*}
\frac{\partial}{\partial w^\alpha} = ((I -\varphi\barvarphi)^{-1})^i{}_j\,b^j{}_\alpha\, T_i.
\end{equation*}
The proof of Theorem \ref{deform2} follows from long computations of\\
$$\partial_w\partial_\barw \log \det g(\partial w^\alpha, \partial w^\barbet)$$
in terms of $z$ coordinates using the above formulae.
\end{proof}

Now we consider the weighted solitons on a Fano manifold $X$, i.e. $c_1(X) > 0$.
We regard $2\pi c_1(X)$ as a K\"ahler class.
The K\"ahler form $\omega$ is expressed as
$$ \omega = \sqrt{-1}\, g_{i\barj}\, dz^i \wedge dz^\barj . $$
Let $T (= \mathbb T_r)$ be a real compact torus in the automorphism group $\Aut (X)$, 
and assume that $\omega$ is $T$-invariant. 
Since $X$ is Fano and simply connected
the $T$-action is Hamiltonian with respect to $\omega$. 
Since the $T$-action naturally lifts to the anti-canonical line bundle
$K_X^{-1}$ we have a canonically normalized moment map $\mu_\omega : X \to \frak t^\ast$,
see Lemma 3.2 in \cite{futaki04}. 
Let
$\Delta := \mu_{\omega}(X)$ be the moment polytope. 
Then $\Delta$ is independent of $\omega \in 2\pi c_1(X)$.
Let $v$ be a positive smooth function on $\Delta$. 
Regarding $\mu$ as coordinates on $\Delta$ using the action angle
coordinates, we may sometimes write $v(\mu)$ instead of $v$. 
The pull-back $\mu^\ast_\omega v$ is a smooth function on $X$, and
for this we write $v(\mu_\omega)=\mu^\ast_\omega v = v \circ \mu_\omega$. 
We say that a K\"ahler metric $\omega$ in $2\pi c_1(X)$ a weighted $v$-soliton or simply $v$-soliton if
$$ \Ric(\omega) - \omega = \sqrt{-1} \partial\barpartial \log v(\mu_\omega)$$ 
where $\Ric(\omega) = -i \partial\barpartial \log \omega^m$ is the Ricci form.
We also call $\omega$ simply a weighted soliton
when it is $v$-soliton for some $v$, or when $v$ is obvious from the context.

\begin{example}
Examples of weighted solitons are as follows (\cite{Inoue19}, \cite{Inoue22}, \cite{Lahdili18}, \cite{ACL21}, \cite{AJL21}, \cite{HanLi23}, 
\cite{LiChi21}).
\begin{enumerate}
\item[(1)] $v(\mu) = e^{\langle\mu,\xi\rangle}$ for some $\xi \in \frak t$ induces a K\"ahler-Ricci soliton,
 \item[(2)] $v(\mu) = \langle\mu, \xi\rangle + a$ for some positive constant $a$ induces a Mabuchi soliton,
\item[(3)] $v(\mu) = \lb \langle\mu,\xi\rangle + a\rb^{-m-2}$, $m=\dim_\bfC X$, induces a Sasaki-Einstein metric on the $U(1)$-bundle of $K_X$.
\end{enumerate}
\end{example}

The Kuranishi family we consider is described (and also has been described is \eqref{Kura}) 
by a family of vector valued $1$-forms 
parametrized by $t \in B$
$$\varphi(t) = \sum_{i=1}^k t^i\varphi_i + \sum_{|I|\ge2}t^I\varphi_I\ \in\ A^{0,1}(T^\prime X)$$
such that 
\begin{equation}\label{Kura2}
\begin{cases}
 \barpartial\varphi(t) = \frac12 [\varphi(t),\varphi(t)];\\
 \barpartial^\ast_f\, \varphi(t) = 0;\\
 \varphi_1, \cdots, \varphi_k\ \text{form a basis of the space of all}\\
 \qquad\qquad T^\prime X\text{-valued}\ \Delta_f\text{-harmonic}\ (0,1)\text{-forms}\\
\end{cases}
\end{equation}
where $\Delta_f = \barpartial_f^\ast\barpartial + \barpartial\,\barpartial_f^\ast$ is the weighted Hodge Laplacian with
$\barpartial_f^\ast$ the formal adjoint of $\barpartial$ with respect to the weighted $L^2$-inner product
$\int_X (\cdot,\cdot) e^f \omega^m$.
Recall by Theorem \ref{deform1}, the K\"ahler form 
$\omega$ on $X_0 = X$ remains to be a K\"ahler form
on $X_t$ for all nearby $t$. 
\begin{theorem}[= Theorem \ref{Main Thm}, \cite{futaki24PAMQ}]
Suppose that $X_0$ has a weighted $v$-soliton. 
Consider the Kuranishi family with $f=\log v(\mu_\omega)$ as above. 
Then, shrinking $B$ if necessary,  the following statements are equivalent.
\begin{enumerate}
\item[(1)] $X_t$ has a weighted $v$-soliton for all $t\in B$.
\item[(2)] $T$ is included in $\Aut(X_t)$, and for the centralizer $\Aut^T(X_t)$ of $T$ in $\Aut(X_t)$, $\dim \Aut^T(X_t) = \dim \Aut^T(X_0)$ for all $t \in B$.
\item[(3)] $T$ is included in $\Aut(X_t)$, and the identity component $\Aut_0^T(X_t)$ of $\Aut^T(X_t)$ is isomorphic to $\Aut_0^T(X_0)$ for all $t \in B$.
\end{enumerate}
\end{theorem}
The case when $X_0$ is a K\"ahler-Einstein manifold and $T = \{1\}$ is due to
Cao-Sun-Yau-Zhang \cite{CaoSunYauZhang2022}, and our proof is
largely parallel to theirs.
\begin{proof}[Outline of the proof.] 
Just as the notion of K\"ahler-Einstein metrics are
generalized 
to constant scalar curvature K\"ahler (cscK for short) metrics 
and further to Calabi extremal K\"ahler metrics \cite{calabi85}, 
the notion of 
weighted solitons are generalized
to Lahdili's weighted cscK metrics \cite{Lahdili18}
and further to weighted extremal metrics. The following three steps are
extensions of the proof in \cite{CaoSunYauZhang2022} for the K\"ahler-Eisntein case to the weighted soliton case.

{\it Step 1.}\ \ 
A weighted extremal metric on $X_0$ extends to nearby $X_t$ 
if the maximal torus $\TildeT \subset \Aut_r^T(X_0)$ acts on $B$ trivially.
c.f. Rollin-Simanca-Tipler \cite{RST}.

{\it Step 2.}\ \ 
The Ricci potential formula in Theorem \ref{deform2} implies that the weighted version of the invariant $\mathrm{Fut}_{v,w}$
given in \eqref{Lich5} below vanishes.
Thus the weighted extremal metric on $X_t$ in Step 1 is
in fact a weighted soliton. Thus the implication 
(2) $\Longrightarrow$ (1) is proved.

{\it Step 3.}\ \ In the K\"ahler-Einstein case of  \cite{CaoSunYauZhang2022}, they show if (3) does not hold then
there is a test configuration for $X_0 = X$ which is not product. Then this implies $X_0$ cannot have a K\"ahler-Einstein
metric by the solution of the Yau-Tian-Donaldson conjecture by \cite{CDS3}, \cite{Tian12}, \cite{DatarSzeke16}, \cite{CSW},
\cite{BBJ}, \cite{LiChi22}, \cite{LXZ22}, i.e. the 
K-polystability characterization of the existence of K\"ahler-Einstein metrics; see the Introduction of the present survey.
Further, these results were extended to weighted solitons by 
 Han-Li \cite{HanLi23}, Blum-Liu-Xu-Zhuang \cite{BLXZ}, Li \cite{LiChi21}. Thus the same proof as \cite{CaoSunYauZhang2022}
 can be applied to the case of weighted solitons. Thus 
(1) $\Longrightarrow$ (3) is obtained.

{\it Step 4.}\ \ 
(3) $\Longrightarrow$ (2) is trivial.
\end{proof}

More details are given in the next two section.

\section{Weighted scalar curvature, or $(v,w)$-scalar curvature by Lahdili.}

Let $X$ be a compact K\"ahler manifold, $\Omega$ its K\"ahler class, and 
 $\Aut_r(X) \subset \Aut(X)$ the reduced automorphism group, 
 i.e. the Lie algebra $\frak h_r(X)$ of $\Aut_r(X)$
consists of holomorphic vector fields with non-empty zeros.
They are obtained in the form $\grad'u$ for $u \in C^\infty_{\bfC}$.
Let $T (= \mathbb T_r)$ be a compact real torus in $\Aut_r(X)$. 
Let $\omega \in \Omega$ be a $T$-invariant K\"ahler form. 
Then $T$ acts on $(X,\omega)$ in the Hamiltonian way. 

Let $\mu_\omega : X \to \frak t^\ast$ be the moment map. 
Then $\Delta:=\mu_\omega(X)$ is a compact convex polytope.
This is independent of $\omega \in \Omega$ up to translation.
Let $v$ be a positive smooth function on $\Delta$. 
Write $v = v(\mu)$ as a function on $\Delta$, 
$v(\mu_\omega) := \mu_\omega^\ast v$ as a positive smooth function on $X$.

Define $v$-scalar curvature $S_v(\omega)$ of a $T$-invariant K\"ahler form $\omega$ by
$$ S_v(\omega):= v(\mu_\omega) S(\omega) + 2 \Delta_\omega v(\mu_\omega) 
+ \langle g_\omega,\mu_\omega^\ast Hess(v)\rangle$$
where $S(\omega)$ denotes the K\"ahler geometers' scalar curvature 
$$ S(\omega) = - g^{i\barj} \frac{\partial^2 }{\partial z^i \partial z^\barj} \log \det (g_{l\barl})$$
of $\omega$, and where the Hessian $Hess(v)$ of $v$ is on $\frak t^\ast$  and
$\mu_\omega(p) = (\mu^1(p), \cdots, \mu^\ell(p))$ with $d\mu^\alpha = i(X^\alpha)\omega$ for a basis $X^1, \cdots, X^\ell$ of $\frak t$ so that
$$\langle g_\omega,\mu_\omega^\ast Hess(v)\rangle = g^{i\barj} v_{\alpha\beta}\mu^\alpha_i\mu^\beta_\barj.$$

Let $w$ be another positive smooth function on $\Delta$.
Define $(v,w)$-scalar curvature $S_{v,w}$ by
$$ S_{v,w} = \frac{S_v}{w(\mu_\omega)}.$$
We say $g$ is weighted cscK if $S_{v,w}$ is constant.

The notion of $(v,w)$-scalar curvature was originally introduced in \cite{Lahdili18} as a generalization of
conformally K\"ahler, Einstein-Maxwell metrics \cite{LeBrun16}, \cite{AM}.
Later it turned out that 
the $(v,w)$-cscK metrics include much more unexpected examples such as weighted solitons \cite{ACL21}, \cite{AJL21}.

We call $g$ a weighted extremal metric or $(v,w)$-extremal metric if 
$$\grad'S_{v,w} = g^{i\barj}\frac{\partial S_{v,w}}{\partial z^\barj}\frac{\partial}{\partial z^i}$$
is a holomorphic vector field.

Define $L_v\varphi$ for complex valued smooth functions $\varphi$ by
$$ L_v \varphi = \nabla^i\nabla^j(v(\mu_\omega)\nabla_i\nabla_j \varphi),$$
and call $L_v$ the $v$-twisted Lichnerowicz operator. Obviously, 
$$ \int_X (L_v\varphi)\, \barpsi\, \omega^m =  \int_X \varphi\, \overline{L_v \psi}\, \omega^m$$
and
$$\Ker\, L_v  \cong \frak h_r(X) \cong Lie(\Aut_r(X)).$$ 
We also define $L_{v,w}$ by
$$ L_{v,w} = \frac 1{w(\mu_\omega)}\, L_v.$$
Then, for $g_{ti\barj} = g_{i\barj} + t\varphi_{i\barj}$ we have
\begin{equation*}\label{Lich1}
\left.\frac d{dt}\right|_{t=0} S_v(g_t) = - L_v \varphi + S_v^i\,\varphi_i,
\end{equation*}
\begin{equation*}\label{Lich2}
\left.\frac d{dt}\right|_{t=0} S_{v,w}(g_t) = - L_{v,w} \varphi + S_{v,w}^i\,\varphi_i.
\end{equation*}
\begin{proposition}\label{Lich3}
A critical point of the weighted Calabi functional 
$$g \mapsto \int_X S_{v,w}^2 (g)\,w(\mu_{\omega})\, \omega^m$$ 
is a weighted extremal metric.
\end{proposition}
\begin{proposition}\label{Lich3.1}
Let $h_X \in \Ker L_v$ be the real Killing potential of $X \in \frak t$, 
i.e. $i \grad'h_X = X'$.
Then $\Fut_v$ and $\Fut_{v,w}$ defined by
\begin{equation}\label{Lich4}
 \Fut_v (X) = \int_X (S_v - c_v)\, h_X\,\omega^m,
 \end{equation}
 and 
 \begin{equation}\label{Lich5}
 \Fut_{v,w} (X) = \int_X (S_{v,w} - c_{v,w})\, h_X\,w(\mu_\omega)\,\omega^m
 \end{equation}
are independent of choice of $\omega \in \Omega$ where 
$c_{v,w} = \int_X S_v\, \omega^m/\int_X w(\mu_\omega)\, \omega^m$ 
and $c_{v} = c_{v,1}$
which
are independent of $\omega \in \Omega$. 
\end{proposition}
\begin{remark}\label{Lich5.1} If $g$ is a $(v,w)$-extremal metric with non-constant $S_{v,w}$ then
$$ \Fut_{v,w}(J\grad S_{v,w}) = \int_X (S_{v,w} - c_{v,w})^2\, w(\mu_\omega)\,\omega^m > 0.$$
\end{remark}
\begin{remark}Calabi's decomposition theorem for extremal K\"ahler metrics holds for weighted extremal K\"ahler metrics.
In particular, the existence of a weighted cscK metric implies 
$$\Aut_{r}^T(X)_{id} = (\Isom_r^T(X)_{id})^\bfC$$
which is a weighted version of Lichnerowicz-Matsushima Theorem.
\end{remark}

Let $\varpi : \mathcal X \to B$ a complex analytic family of complex deformations with $X_0 = X$ where $B$ is an open set in $\bfC^k$ containing $0$ and we put $X_t :=\varpi^{-1}(t)$. 
We assume $(X_0,g_0)=(X,g)$ is a compact $(v,w)$-extremal K\"ahler manifold.
Let $\TildeT$ be the maximal torus in $\Aut_r (X)$ including $T$. 
Suppose that $\TildeT$ acts holomorphically on $\mathcal X \to B$ and trivially on $B$.
Thus $\TildeT$ acts on $X_t$ holomorphically for each $t \in B$.
By the rigidity theorem of Kodaira-Spencer, $X_t$ is K\"ahler for all small $t$. 
Let $\Omega_t$ be a smooth family of K\"ahler classes of $X_t$.
\begin{proposition}[\cite{futaki24PAMQ}, also \cite{RST}, \cite{lebrunsimanca94}]\label{extension}
In the above situation, by shrinking $B$
to a sufficiently small neighborhood of the origin if necessary, for arbitrary small perturbations 
$\Omega_t$ of the K\"ahler class $\Omega=\Omega_0$, there are weighted extremal metrics $g_t$ in $\Omega_t$.
\end{proposition}

\section{Weighted solitons on Fano manifolds.}
Let $X$ be a Fano manifold, and $\omega \in 2\pi c_1(X)$ be a K\"ahler form.
\begin{definition}
Let $v$ be a positive smooth function on the moment map image $\Delta=\mu_\omega(X)$.
We say that $\omega$ is a weighted $v$-soliton (or simply weighted soliton, also $v$-soliton) if
$$ \Ric(\omega) - \omega = i\partial\barpartial \log v(\mu_\omega).$$
\end{definition}
A $T$-invariant K\"ahler form $\omega \in 2\pi c_1(X)$ is a $v$-soliton
if and only if
$\omega$ is $S_{v,w} = 1$ metric with 
\begin{equation*}\label{soliton0}
w(\mu) = (m+\langle d\log v,\mu\rangle)v(\mu). 
 \end{equation*}
 We keep this choice of $w(\mu)$ in what follows.
This can be seen from the formula
\begin{equation}\label{soliton1}
 S_v - w(\mu_\omega) = v(\mu_\omega)\Delta_v(\log v(\mu_\omega) - f)
 \end{equation}
where $f \in C^\infty(X)$ is the Ricci potential of $\omega$, i.e. $ \Ric(\omega) - \omega = i\partial\barpartial f$, and $\Delta_v = v^{-1}\circ \barpartial^\ast \circ v \circ \barpartial$, i.e.
$v(\mu_\omega)$-weighted Laplacian.
By \eqref{soliton1}, we have
$$ \int_X (S_v - w(\mu_\omega)) \omega^m = 0,$$
and thus $c_{v,w} = 1$ and 
 \begin{eqnarray*}
 \Fut_{v,w} (X) = \int_X (S_v - w(\mu_\omega))\, h_X\,\omega^m.
 \end{eqnarray*}
 Using \eqref{soliton1} this can be rewritten as 
 \begin{equation*}
  \Fut_{v,w}(X) = \int_X (JX)(\log v(\mu_\omega) - f)\,v(\mu_\omega)\,\omega^m. \label{Lich6}
 \end{equation*}

Now, let $\varpi : \mathfrak X \to B$ be the Kuranishi family of a Fano manifold $X$. 
By Theorem \ref{deform1} and Theorem \ref{deform2}, we know the following.

(1) The K\"ahler form $\omega$ on $X_0 = X$ remains to be K\"ahler forms on $X_t$ for all $t\in B$.

(2) The Ricci form $\Ric(X_t,\omega)$
of $(X_t, \omega)$ is given by
\begin{eqnarray*}
 \Ric(X_t,\omega) 
 = \omega + \partial_t\barpartial_t (f_0 + \log\det(I - \varphi(t)\overline{\varphi(t)})).
 \end{eqnarray*}
 But since we assume $(X_0, \omega_0)$ with $\omega_0 = \omega$ is a $v$-soliton we have
 \begin{eqnarray*}
 f_0 = \log v(\mu_{\omega_0}).
 \end{eqnarray*}
 Hence we have
 \begin{equation} \label{Lich7}
  \Fut_{v,w}(t)(X) = - \int_X (JX)(\log \det(I - \varphi(t)\overline{\varphi(t)}))\,v(\mu_\omega)\,\omega^m.
 \end{equation}
But since any automorphism of $X_t$ preserves $\varphi(t)$ the derivative by $JX$ on the right hand side of \eqref{Lich7} vanishes.
Thus $\Fut_{v,w}(t)$ vanishes.
\begin{lemma}[\cite{futaki24PAMQ}, also \cite{CaoSunYauZhang2022}]
Suppose that (2) of the Main Theorem is satisfied ($\dim \Aut^T(X_t) = \dim\Aut^T(X_0)$). Then the identity component $\Aut^T_0(X)$ of $\Aut^T(X)$ 
acts on $H^1(X_0, T'X_0)\cong T'_0B$ trivially, and hence on $B$ trivially.
\end{lemma}
Hence, we can apply Proposition \ref{extension} to obtain weighted extremal metrics on $X_t$.
But since $\Fut_{v,w}(t) = 0$ on $X_t$ as we saw above, they are in fact weighted cscK, i.e. weighted solitons.
This is some more detail of the proof of (2) $\Longrightarrow$ (1) for Theorem \ref{Main Thm}.

Next we give some more detail of the proof that (1) implies (3). 
We first show the action of $G := \Isom_0^T(X_0,\omega)$ on $B$ is trivial. 
Suppose this is not the case.
Then by the same arguments as in \cite{CaoSunYauZhang2022}, there is a  non-product test configuration for
 $(X_t, K_{X_t}^{-k})$. 
 But this is impossible since $X_t$ has a $v$-soliton and K-polystable, the central fiber $X_0$ also has a
 $v$-soliton and the Donaldson-Futaki invariant is zero, see Theorem 1.17 and 1.21 in \cite{LiChi21}.

Thus the action of $G$ on $B$ is trivial. 
Then $G$ preserves both $\omega$ and $\varphi(t)$ of the Kuranishi family, 
and thus we have an inclusion
$G \subset \Isom_0(X_t,\omega)$. In particular $T \subset \Isom_0(X_t,\omega) \subset \Aut_0(X_t)$ 
and 
$G \subset \Isom_0^T(X_t,\omega)$. Hence
$\Aut_0^T(X_0) = G^\bfC \subset \Aut_0^T(X_t)$. 

But since $\dim H^0(X_t, T^\prime X_t)^T$ is upper semi-continuous we obtain
$G^\bfC = \Aut^T(X_t)$ for all $t \in B$. 
This proves that (1) implies (3). That (3) implies (2) is trivial. This completes the proof
of Main Theorem.

\end{document}